\documentclass[10pt]{article}
\usepackage{authblk}

\usepackage{margincomment}
\usepackage{xcolor}
\definecolor{azure}{rgb}{0.25, 0.41, 0.88}
\definecolor{ballblue}{rgb}{0.13, 0.67, 0.8}
\definecolor{mountainmeadow}{rgb}{0.31, 0.78, 0.47}
\definecolor{blush}{rgb}{0.87, 0.36, 0.51}

\usepackage{times}
\usepackage{mathtools}
\usepackage[
  colorlinks   = true,   
  linkcolor    = ballblue,  
  citecolor    = ballblue, 
  urlcolor     = ballblue,  
  linktoc      = page    
]{hyperref}
\usepackage[shortlabels]{enumitem}
\usepackage{mathrsfs}
\usepackage{mhequ}

\RequirePackage{color}
\RequirePackage[a4paper,top=1.6in,bottom=1.6in,left=1.6in,right=1.6in]{geometry}
\RequirePackage{lastpage}
\usepackage[english]{babel}
\RequirePackage{amsmath,amsthm,amssymb,mathtools}
\RequirePackage{enumitem}
\RequirePackage{placeins}

\usepackage{fancyhdr}
\usepackage[font=small]{caption}
\usepackage{titlesec}
\titleformat{\section}
  {\normalfont\fontsize{14}{15}\bfseries}{\thesection}{1em}{}
\titleformat{\subsection}
  {\normalfont\fontsize{12}{15}\bfseries}{\thesubsection}{1em}{}

\newcommand{\expected}{\mathbb{E}}
\newcommand{\prob}{\mathbb{P}}
\newcommand{\eps}{\varepsilon}

\newcommand{\energy}{\mathscr{E}}

\makeatletter
\DeclareFontFamily{OMX}{MnSymbolE}{}
\DeclareSymbolFont{MnLargeSymbols}{OMX}{MnSymbolE}{m}{n}
\SetSymbolFont{MnLargeSymbols}{bold}{OMX}{MnSymbolE}{b}{n}
\DeclareFontShape{OMX}{MnSymbolE}{m}{n}{
    <-6>  MnSymbolE5
   <6-7>  MnSymbolE6
   <7-8>  MnSymbolE7
   <8-9>  MnSymbolE8
   <9-10> MnSymbolE9
  <10-12> MnSymbolE10
  <12->   MnSymbolE12
}{}
\DeclareFontShape{OMX}{MnSymbolE}{b}{n}{
    <-6>  MnSymbolE-Bold5
   <6-7>  MnSymbolE-Bold6
   <7-8>  MnSymbolE-Bold7
   <8-9>  MnSymbolE-Bold8
   <9-10> MnSymbolE-Bold9
  <10-12> MnSymbolE-Bold10
  <12->   MnSymbolE-Bold12
}{}

\let\llangle\@undefined
\let\rrangle\@undefined
\DeclareMathDelimiter{\llangle}{\mathopen}
                     {MnLargeSymbols}{'164}{MnLargeSymbols}{'164}
\DeclareMathDelimiter{\rrangle}{\mathclose}%
                     {MnLargeSymbols}{'171}{MnLargeSymbols}{'171}
\makeatother

\DeclareMathOperator*{\mymax}{\text{\normalfont max}}
\DeclareMathOperator*{\mymin}{\text{\normalfont min}}
\DeclareMathOperator*{\mysup}{\text{\normalfont sup}}

\DeclareMathOperator*{\mylim}{\text{\normalfont lim}}
\DeclareMathOperator*{\mylog}{\text{\normalfont log}}
\DeclareMathOperator*{\mylimsup}{\text{\normalfont limsup}}
\DeclareMathOperator*{\myliminf}{\text{\normalfont liminf}}
\DeclareMathOperator*{\myexp}{\text{\normalfont exp}}
\DeclareMathOperator*{\de}{\hspace{-.1em}\text{\normalfont d\hspace{-.15em}}}

\numberwithin{equation}{section}
\newtheorem{theorem}{Theorem}[section]

\newtheorem{lemma}[theorem]{Lemma}
\newtheorem{proposition}[theorem]{Proposition}
\theoremstyle{definition}
\newtheorem{definition}[theorem]{Definition}
\newtheorem{remark}[theorem]{Remark}
\newtheorem*{acknowledgements}{Acknowledgements}

\DeclareFontEncoding{LS1}{}{}
\DeclareFontSubstitution{LS1}{stix}{m}{n}
\DeclareSymbolFont{stixletters}{LS1}{stix}{m}{it}
\DeclareMathAccent{\cev}{\mathord}{stixletters}{"91}
\DeclareMathAccent{\vec}{\mathord}{stixletters}{"92}

\usepackage{cancel}

\makeatletter
\newcommand\@affiliationlist{}

\renewcommand\affil[2][]{%
  \g@addto@macro\@affiliationlist{%
    \item[\textsuperscript{\scriptsize #1}] #2%
  }%
}

\newlength{\affillabelwidth}

\renewcommand{\maketitle}{
  \vspace*{0em}%
  \begin{center}{\LARGE\bfseries \@title \par}\end{center}
  \vspace{1em}%
  \begin{center}{\small \scshape\@date \par}\end{center}\par 
  \vspace{4em}
  \begin{list}{}{\setlength{\leftmargin}{1.5em}\setlength{\rightmargin}{1.5em}}
    \item[] {\raggedright\large\@author \par}
  \end{list}

  \settowidth{\affillabelwidth}{\textsuperscript{\scriptsize 8}} 
  \begin{list}{}{%
      \setlength{\leftmargin}{2.2em}
      \setlength{\rightmargin}{1.5em}
      \setlength{\labelwidth}{\affillabelwidth}
      \setlength{\labelsep}{0.2em}
      \setlength{\itemsep}{.2em}
      \setlength{\parsep}{0pt}
      \setlength{\topsep}{0pt}
  }
    {\raggedright\small \@affiliationlist}
  \end{list}\vspace{0.5em}
}

\renewenvironment{abstract}
  {\par                               
   \begin{list}{}{                        
      \setlength{\leftmargin}{1.5em}
      \setlength{\rightmargin}{1.5em}}
   \item[]                                   
   {\bfseries\abstractname}\par               
   \small\noindent\ignorespaces}             
  {\end{list}}

\begin{document}
\allowdisplaybreaks
\date{\today}
\title{Energy Estimate and Convergence of Weak Solutions of a Class of Nonlinear PDEs with Robin Boundary Conditions}

\author{\bfseries O. Araújo\textsuperscript{1}, P. Gonçalves\textsuperscript{2}, A. Neumann\textsuperscript{3}, M. C. Ricciuti\textsuperscript{4}
}

\affil[1]{Department of Mathematics, Universidade Federal da Paraíba, Cidade Universitária, 58059-900, João Pessoa, PB, Brazil, Email: \href{mailto:oslenne.nogueira@ccen.ufpb.br}{\texttt{oslenne.nogueira@ccen.ufpb.br}}}

\affil[2]{Center for Mathematical Analysis, Geometry and Dynamical Systems, Instituto Superior Técnico, Universidade de Lisboa, Av. Rovisco Pais, 1049-001 Lisboa, Portugal, Email: \href{mailto:pgoncalves@tecnico.ulisboa.pt}{\texttt{pgoncalves@tecnico.ulisboa.pt}}}
 
\affil[3]{Mathematics and Statistics Institute, Universidade Federal do Rio Grande do Sul, Av. Bento Gonçalves, 9500, Porto Alegre, RS, Brazil, Email: \href{mailto:aneumann@mat.ufrgs.br}{\texttt{aneumann@mat.ufrgs.br}}}

\affil[4]{Department of Mathematics, Imperial College London, 180 Queen's Gate, London SW7 2AZ, United Kingdom, Email: \href{mailto:maria.ricciuti18@imperial.ac.uk}{\texttt{maria.ricciuti18@imperial.ac.uk}}}

\clearpage\maketitle
\thispagestyle{empty}

\begin{abstract}
We establish strong convergence in $L^2([0,T]\times[0,1])$ of weak solutions to a class of nonlinear PDEs with Robin boundary conditions to weak solutions of the corresponding Neumann and Dirichlet problems as the boundary parameter $\kappa$ tends to $0$ and $\infty$, respectively. Our approach relies on a particle system 
whose hydrodynamic limits yield the Robin problem for each fixed $\kappa>0$. An energy estimate derived from the microscopic stochastic dynamics provides the key bounds for passing to the subsequent limits $\kappa\to0$ and $\kappa\to\infty$, thereby obtaining the Neumann and Dirichlet problems.
\end{abstract}

\setcounter{tocdepth}{1}


\section{Introduction}

Boundary conditions play a fundamental role in the analysis of diffusive equations and in the description of systems interacting with their environment. Robin boundary conditions provide a natural interpolation between Neumann and Dirichlet boundary conditions, with the corresponding parameter measuring the strength of the interaction between the system and the external reservoirs. In this paper, we study this interpolation for a class of nonlinear diffusion equations that we define below,  and establish the convergence of their weak solutions as the strength of the boundary interaction vanishes or becomes arbitrarily large.

Let $\alpha, \beta, \lambda, \delta>0$, and let $\Phi:[0, \infty)\to[0, \infty)$ be a  strictly increasing and globally Lipschitz function. We consider the nonlinear PDE with Robin boundary conditions given by
\begin{equation}\label{eq:zre_robin}
    \left\{ 
    \begin{array}{ll}
    \partial_t \rho^\kappa_t(u) = \Delta \Phi(\rho^\kappa_t(u)) &\mbox{ for } u \in (0,1) \mbox{ and } t \in (0,T],
    \\\partial_u \Phi (\rho^\kappa_t(0)) = \kappa(\lambda \Phi(\rho^\kappa_t(0))-\alpha) &\mbox{ for } t \in (0,T],
    \\\partial_u \Phi (\rho^\kappa_t(1)) = \kappa(\beta-\delta\Phi(\rho^\kappa_t(1))) &\mbox{ for } t \in (0,T],
    \\\rho^\kappa_0(u) = \gamma(u) &\mbox{ for } u \in [0,1],
    \end{array}
    \right.
\end{equation}
where $\gamma:[0, 1]\to[0, \infty)$ is an integrable function.

Our main goal is to understand the dependence of the weak solution
$\rho^\kappa$ on the parameter $\kappa$. For every measurable initial
profile $\gamma:[0,1]\to[0,\infty)$ satisfying
\eqref{eq:initial_condition}, we prove the strong convergence
\[
    \lim_{\kappa\to0}
    \|\rho^\kappa-\rho^0\|_{L^2([0,T]\times[0,1])}=0
\]
and
\[
    \lim_{\kappa\to\infty}
    \|\rho^\kappa-\rho^\infty\|_{L^2([0,T]\times[0,1])}=0,
\]
where $\rho^0$ and $\rho^\infty$ are the unique weak solutions
to the corresponding Neumann and Dirichlet problems, respectively,
with the same initial profile $\gamma$, namely
\begin{equation}\label{eq:zre_neumann}
    \left\{ 
    \begin{array}{ll}
    \partial_t \rho^0_t(u) = \Delta \Phi(\rho^0_t(u)) &\mbox{ for } u \in (0,1) \mbox{ and } t \in (0,T],
    \\\partial_u \Phi (\rho^0_t(0)) = 0 &\mbox{ for } t \in (0,T],
    \\\partial_u \Phi (\rho^0_t(1)) = 0 &\mbox{ for } t \in (0,T],
    \\\rho^0_0(u) = \gamma(u) &\mbox{ for } u \in [0,1]
    \end{array}
    \right.
\end{equation}
and
\begin{equation}\label{eq:zre_dirichlet}
    \left\{ 
    \begin{array}{ll}
    \partial_t \rho^\infty_t(u) = \Delta \Phi(\rho^\infty_t(u)) &\mbox{ for } u \in (0,1) \mbox{ and } t \in (0,T],
    \\\Phi (\rho^\infty_t(0)) = \frac{\alpha}{\lambda} &\mbox{ for } t \in (0,T],
    \\\Phi (\rho^\infty_t(1)) = \frac{\beta}{\delta} &\mbox{ for } t \in (0,T],
    \\\rho^\infty_0(u) = \gamma(u) &\mbox{ for } u \in [0,1],
    \end{array}
    \right.
\end{equation}
respectively. 
The limiting boundary conditions can be understood directly from the Robin boundary conditions. When $\kappa\to0$, the interaction with the reservoirs becomes weaker and the boundary flux vanishes, formally yielding the Neumann conditions. On the other hand, as $\kappa\to\infty$, the Robin conditions force
$$
\lambda\Phi(\rho^\kappa_t(0))-\alpha\to_{\kappa\to\infty}0,
\qquad
\beta-\delta\Phi(\rho^\kappa_t(1))\to_{\kappa\to\infty}0,
$$
which formally gives
$$
\Phi(\rho_t(0))=\frac{\alpha}{\lambda},
\qquad
\Phi(\rho_t(1))=\frac{\beta}{\delta}.
$$
Thus, in the strong-reservoir regime, the boundary values of the diffusive system are determined by the reservoirs, leading to  Dirichlet boundary conditions.

Although this heuristic argument is straightforward, proving the rigorous convergence of the corresponding weak solutions is a substantially more difficult task. The main issue is to obtain estimates that are uniform in $\kappa$ and provide sufficient compactness to pass to the limit both as $\kappa\to0$ and as $\kappa\to\infty$. Since the equation is nonlinear, the relevant quantities are $\Phi(\rho^\kappa)$ and its spatial derivative. Therefore, a uniform energy estimate gives precisely the required control and, moreover, it provides the information needed to identify the limiting boundary conditions as  Dirichlet or Neumann. The energy estimate is therefore the central ingredient in connecting the Robin boundary conditions to  the Neumann and Dirichlet limits.

Our approach to obtaining this estimate is  probabilistic. We consider boundary-driven symmetric zero-range processes whose hydrodynamic equations are given by \eqref{eq:zre_robin}. The hydrodynamic limit for these systems was established in \cite{agnr25}. The nonlinear function $\Phi$ appearing in the macroscopic equation is determined by the jump rates of the underlying zero-range processes, and the  assumptions on these rates impose restrictions on the class of nonlinearities that can be treated by our method.

The advantage of this microscopic representation is that the required energy estimate can be derived from the random microscopic dynamics. This is particularly delicate in the present setting because the occupation variables of a zero-range process are unbounded. Consequently, obtaining estimates that are uniform in $\kappa$ requires suitable moment estimates and replacement arguments. This distinguishes our approach from related works such as \cite{FGN14,dgn21}, where the underlying microscopic dynamics were of exclusion type so that  occupation variables are bounded and therefore,  the corresponding energy estimates are simpler.

Let us briefly describe the class of zero-range processes for which our arguments hold. The generator is given by \eqref{eq:generator} and has jump rates given by a function
$g:\mathbb{N}\to[0,\infty)$ satisfying $g(0)=0$, $g(k)>0$ for $k\geq1$, and
\begin{equation}\label{eq}
g^*:=\sup_{k\geq0}|g(k+1)-g(k)|<\infty.
\end{equation}
We also assume a spectral-gap estimate related to the mixing properties of the process. If $\operatorname{gap}(j,\ell)$ denotes the spectral gap of the generator of the zero-range process on $\{1,\ldots,\ell\}$ with $j$ particles and jump rate $g$, we assume that there exists $C=C(g)>0$ such that
\begin{equation}\label{spectral_gap}
\operatorname{gap}(j,\ell)
\geq
\frac{C}{\ell^2(1+j/\ell)^2}
\end{equation}
for all $\ell\geq2$ and $j\geq0$. The precise definition of the rates and of the microscopic dynamics is given in Section \ref{sec:zrp}.

The hydrodynamic limit established in \cite{agnr25} identifies the macroscopic density of these processes with the unique weak solution of the nonlinear diffusion equation with Robin boundary conditions, with the precise limiting boundary condition depending on the strength of the reservoirs. In the proof of that result, suitable energy estimates are already required to establish the Sobolev regularity of the limiting density. This is necessary in order to obtain uniqueness of the weak solution to \eqref{eq:zre_robin}.  In the present work, we exploit the same microscopic dynamics to derive a stronger energy estimate, uniform with respect to the parameter $\kappa$. This additional control allows us to study the dependence of the macroscopic solution on the reservoir strength and, in particular, to prove its convergence to the Neumann and Dirichlet limits.

The main contributions of this paper are twofold. First, using the underlying zero-range dynamics, we derive a uniform energy estimate for the family $\{\rho^\kappa\}_{\kappa>0}$. This estimate provides sufficient control to establish the boundary behaviour of the weak solutions and to identify the Robin boundary conditions in the appropriate sense. Second, we use the resulting properties of the associated energy functional to prove the strong convergence of $\rho^\kappa$ in $L^2([0,T]\times[0,1])$ as $\kappa\to0$ and $\kappa\to\infty$. In particular, our analysis gives a rigorous derivation of the transition from Robin to Neumann and Dirichlet boundary conditions in the two extreme regimes.
Our strategy is related to the approach developed for the heat equation with linear Robin boundary conditions in \cite{FGN14} and for the porous medium equation with nonlinear Robin boundary conditions in \cite{dgn21}.  In both cases, the underlying microscopic systems have bounded occupation variables. By contrast, the zero-range processes considered here allow arbitrarily many particles at each site,
which makes the derivation of an energy estimate uniform in
$\kappa$ substantially more delicate. Our framework also
accommodates a broader class of nonlinearities than the
polynomial ones considered in \cite{dgn21}. Addressing these
additional difficulties requires new probabilistic estimates,
whose derivation constitutes a central technical contribution
of this work. The extension of this strategy to broader classes of interacting particle systems, such as the general class of misanthrope processes considered in \cite{GSaa}, is left for future work.

\medskip 

\noindent\textbf{Outline.} 

In Section \ref{sec:notation_results}, we introduce the class of functions $\Phi$ for which our results hold, define the corresponding notion of weak solution, and state the main results of the paper, including the uniform energy estimate, the properties of the associated energy functional, and the convergence of weak solutions as $\kappa\to0$ and $\kappa\to\infty$. In Section \ref{sec:energy}, we recall from \cite{agnr25} the definition of the boundary-driven zero-range processes and the  results needed in the proofs, and we establish the uniform energy estimate. Section \ref{sec:properties} is devoted to the analysis of the energy functional and its relevant properties. In Section \ref{sec:convergence}, we prove the convergence of the weak solutions of \eqref{eq:zre_robin} in the regimes $\kappa\to0$ and $\kappa\to\infty$. Finally, Appendix \ref{sec:appendix} collects auxiliary results concerning weak solutions of the nonlinear PDE in \eqref{eq:zre_robin}, including uniqueness, Hölder continuity in time, and a comparison principle.

\section{Notation and Results}\label{sec:notation_results}
Fix a time horizon $T>0$. Let $C^{m,n}([0, T]\times[0, 1])$ denote the set of continuous functions on $[0, T]\times[0, 1]$ which are $m$ times differentiable in the first variable and $n$ in the second, with continuous derivatives, where $m$ and $n$ are non negative integers. We will also denote by $\mathcal{H}^1$ the Sobolev space $\mathcal H^1(0, 1)$, equipped with the norm
\begin{equation*}
    \|\cdot\|_{\mathcal{H}^1}^2:=\|\cdot\|^2_{L^2((0, 1))}+\|\nabla\cdot \|^2_{L^2((0, 1))},
\end{equation*}
and by ${L^2(0,T; \mathcal{H}^1)}$ the set of measurable functions $f:[0, T]\to\mathcal{H}^1$ with $\int_0^T\|f_t\|_{\mathcal{H}^1}^2\de t<\infty$. Moreover, let $\langle \cdot, \cdot\rangle$ denote the inner product $\langle f, h\rangle:=\int_0^1 f(u)h(u)\de u$ on the set of measurable functions in $L^2([0,1])$. Finally, for each $p\in[1, \infty]$, we will also denote by $\|\cdot\|_p$ the norms $\|\cdot\|_{L^p([0, 1])}$ and  $\|\cdot\|_{L^p([0, T]\times[0, 1])}$; the meaning will be clear from context.

\subsection{Assumptions on $\boldsymbol{\Phi}$} \label{sec:assumptions_Phi}

The class of functions $\Phi$ that we consider in the PDE \eqref{eq:zre_robin} 
is constructed from a corresponding class of functions $g :\mathbb{N}\to [0, \infty)$, see \eqref{eq:R} and \eqref{eq:Phi}. We impose the following assumptions on this function $g$ to
ensure that the associated nonlinearity $\Phi$ is strictly
increasing and globally Lipschitz.

Let $g :\mathbb{N}\to [0, \infty)$ be a map such that $g(0)=0$, it is strictly positive on the set of positive integers, and it has uniformly bounded increments in the following sense:
\begin{equation}\label{eq:function_g}
    g^*:=\mysup_{j\ge 0}|g(j+1)-g(j)|<\infty.
\end{equation}
In addition, we require $g$ to satisfy the two conditions below. The first one is very simple:

\medskip
\textsc{\bfseries Assumption (ND).} The function $g$ is nondecreasing.

\medskip
The second one, instead, is more technical, and is related to the mixing properties of some specific discrete stochastic dynamics, namely those of a \textit{zero-range process}. For an introduction to this process, see \cite[Section 2.3]{kl99}. 

\medskip
\textsc{\bfseries Assumption (SG).} Let $\text{gap}(j, \ell)$ denote the spectral gap of the generator of a symmetric, nearest-neighbour zero-range process on $\{1, \ldots, \ell\}$ with $j$ particles and jump rate $g$. Then, the function $g$  satisfies the following:
\begin{equation*}
    \text{there exists } C=C(g)>0 \text{ such that gap}  (j, \ell)\ge \frac{C}{\ell^2(1+\frac{j}{\ell})^2}\text{ for all } \ell\ge2 \text{ and } j\ge0.
\end{equation*}
Examples of functions $g$ for which \textsc{Assumption (SG)} holds include:
\begin{enumerate}[i)]  
    \item $g(k)=k^\gamma$ for $0<\gamma<1$ \cite{nag10};
    \item $g(k)=\boldsymbol{1}_{[1, \infty)}(k)$ \cite{mor06};
    \item $g$ satisfying \eqref{eq:function_g} and such that there exist $M\in\mathbb{N}$ and $a>0$ such that $g(k)-g(l)\ge a$ for all $k\ge l+M$ \cite{lsv96}.
\end{enumerate}

Let now $g(j)!:=\prod_{1\le k\le j} g(k)$, with $g(0)!:=1$, and define the function $Z:[0, \varphi^*)\to[0, \infty)$ as 
\begin{equation*}
    Z(\varphi):=\sum_{j=0}^\infty\frac{\varphi^j}{g(j)!},
\end{equation*} 
where $\varphi^*$ is its radius of convergence. Let $R:[0,\varphi^*)\to[0, \infty)$ be the function defined as the power series
\begin{equation}\label{eq:R}
    R(\varphi):= \frac{1}{Z(\varphi)}\sum_{j=0}^\infty j\frac{\varphi^j}{g(j)!}.
\end{equation}
We assume that $\mylim_{\varphi\uparrow\varphi*}Z(\varphi)=\infty$: then, as proved in \cite[Section 2.3]{kl99}, the range of $R$ is $[0, \infty)$ and $R$ is strictly increasing, so that in particular it is invertible. Then, for $g$ satisfying the assumptions above, we define
\begin{equation}\label{eq:Phi}
    \Phi:=R^{-1}:[0, \infty)\to[0, \varphi^*)
\end{equation} 
In particular, $\Phi$ is strictly increasing.  Moreover, by \cite[Corollary 2.3.6]{kl99}, we have that $\Phi$ is globally Lipschitz on $[0, \infty)$ with constant $g^*$, namely
\begin{equation}\label{eq:Phi_g*}
   \left|\Phi(u)-\Phi(v)\right| \le g^*|u-v|
\end{equation}
for all $u, v\in[0, \infty)$.

Throughout this paper, we assume that the reservoir parameters
$\alpha,\beta,\lambda,\delta>0$ satisfy
\begin{equation}\label{eq:reservoir_assumption}
    \max\left\{\frac{\alpha}{\lambda},
               \frac{\beta}{\delta}\right\}<\varphi^*,
\end{equation}
where $\varphi^*$ is the radius of convergence of the partition
function $Z$. In particular, the reservoir densities
$\Phi^{-1}(\alpha/\lambda)$ and $\Phi^{-1}(\beta/\delta)$
are well-defined and finite.

\subsection{The Nonlinear PDE}
Now we introduce the precise notion of weak solutions to \eqref{eq:zre_robin}. 
\begin{definition}[Nonlinear PDE with Robin Boundary Conditions]\label{def:weak_robin}  Let $\kappa>0$ and let $\gamma:[0,1]\to[0, \infty)$ be a measurable and integrable function. We say that $\rho^\kappa:[0, T]\times [0, 1]\to[0, \infty)$ is a weak solution of the nonlinear PDE with Robin boundary conditions  given in \eqref{eq:zre_robin} if:
\begin{enumerate}[i)]
\item $\rho^\kappa \in L^2([0, T]\times [0, 1])$,
\item $\Phi(\rho^\kappa)\in L^2([0,T], \mathcal H^1)$,
\item for any $t\in[0,T]$ and $G\in C^{1,2}([0,T]\times [0,1])$,
\begin{equation} \label{eq:integral_zre_rob}
    \begin{split}
    \langle\rho^\kappa_t, G_t\rangle &- \langle\gamma, G_0\rangle - \int_0^t\{\langle \rho^\kappa_s, \partial_s G_s\rangle + \langle\Phi(\rho^\kappa_s), \Delta G_s\rangle\}\de s
    \\&+\int_0^t \{\Phi(\rho^\kappa_s(1))\partial_u G_s(1)-\Phi(\rho^\kappa_s(0))\partial_u G_s(0) \}\de s
    \\&-\kappa\int_0^t \big\{(\beta-\delta\Phi(\rho^\kappa_s(1)))G_s(1) - (\lambda \Phi(\rho^\kappa_s(0))-\alpha)G_s(0) \big\}\de s = 0.
    \end{split}
\end{equation}
\end{enumerate}
\end{definition}
We aim to take limits in the parameter \(\kappa\) and obtain weak solutions to the corresponding equations with Neumann and Dirichlet boundary conditions. We now give their definitions.
\begin{definition}[Nonlinear PDE with Neumann Boundary Conditions]\label{def:weak_neumann}  Let $\gamma:[0,1]\to[0, \infty)$ be a measurable and integrable function. We say that $\rho:[0, T]\times [0, 1]\to[0, \infty)$ is a weak solution of the nonlinear PDE with Neumann boundary conditions \eqref{eq:zre_neumann} if:
\begin{enumerate}[i)]
\item $\rho\in L^2([0, T]\times [0, 1])$,
\item $\Phi(\rho)\in L^2([0,T], \mathcal H^1)$,
\item for any $t\in[0,T]$ and $G\in C^{1,2}([0,T]\times [0,1])$, $\rho$ satisfies the integral formulation \eqref{eq:integral_zre_rob} with $\kappa=0$.
\end{enumerate}
\end{definition}

\begin{definition}[Nonlinear PDE with Dirichlet Boundary Conditions]\label{def:weak_dirichlet}  Let $\gamma:[0,1]\to[0, \infty)$ be a measurable and integrable function. We say that $\rho:[0, T]\times [0, 1]\to[0, \infty)$ is a weak solution of the nonlinear PDE with Dirichlet boundary conditions \eqref{eq:zre_dirichlet} if:
\begin{enumerate}[i)]
\item $\rho\in L^2([0, T]\times [0, 1])$,
\item $\Phi(\rho)\in L^2([0,T], \mathcal H^1)$,
\item for any $t\in[0,T]$ and $G\in C_0^{1,2}([0,T]\times [0,1])$,
\begin{equation} \label{eq:integral_zre_neu}
    \begin{split}
    \langle\rho_t, G_t\rangle &- \langle\gamma, G_0\rangle - \int_0^t\{\langle \rho_s, \partial_s G_s\rangle + \langle\Phi(\rho_s), \Delta G_s\rangle\}\de s
    \\&+\int_0^t \left\{\frac{\beta}{\delta}\partial_u G_s(1)-\frac{\alpha}{\lambda}\partial_u G_s(0) \right\}\de s=0,
    \end{split}
\end{equation}
\item $\Phi(\rho_t(0))=\frac{\alpha}{\lambda}$ and $\Phi(\rho_t(1))=\frac{\beta}{\delta}$ for almost every $t\in(0, T]$.
\end{enumerate}
\end{definition}

The proof of the uniqueness of the weak solutions of \eqref{eq:zre_robin}, \eqref{eq:zre_neumann} and \eqref{eq:zre_dirichlet}, in the sense of the definitions above, is given in Proposition \ref{prop:uniqueness_zre}. 

\subsection{Main Results}\label{subsec:main_results}
For $\kappa>0$, define the function $\bar\varphi^\kappa:[0, 1]\to[0, \infty)$ as
\begin{equation}\label{eq:limit_fug}
    \bar\varphi^\kappa(u):=\frac{-\kappa(\alpha\delta-\beta\lambda)u+\kappa\alpha\delta+\alpha+\beta}{\kappa\lambda\delta+\lambda+\delta}.
\end{equation}

\begin{remark}\label{remark:bar_rho_kappa}
    It is not hard to check that $\bar\rho^\kappa(u):=\Phi^{-1}(\bar\varphi^\kappa(u))$ is a stationary solution to \eqref{eq:zre_robin}.
\end{remark}

Given $\kappa>0$ and a measurable function $\xi:[0, 1]\to[0, \infty)$, we define the measure $W_{\kappa, \xi}$ on $[0, 1]$ via
\begin{equation}
    W_{\kappa, \xi}(\de u):=(\Phi(\xi(u))+\bar\varphi^\kappa(u))\de u+\frac{\Phi(\xi(0))+\bar\varphi^\kappa(0)}{\kappa\lambda}\delta_0(\de u)+\frac{\Phi(\xi(1))+\bar\varphi^\kappa(1)}{\kappa\delta}\delta_1(\de u),
\end{equation}
where $\delta_z(\de u)$ denotes the Dirac measure. Moreover, we denote by $\mathscr{B}$ the subspace of {{bounded and}} measurable functions $\xi:[0, T]\times[0, 1]\to[0, \infty)$ such that $s\mapsto\xi_s(0)$ and $s\mapsto\xi_s(1)$ are measurable and bounded. Given a function $\xi\in\mathscr{B}$, let $\llangle \cdot, \cdot\rrangle_{\kappa, \xi}$ denote the inner product on $L^2([0, T]\times[0,1])$ given by
\begin{align*}
    \llangle H, G\rrangle_{\kappa, \xi} &:=\int_0^T \int_0^1 H_s(u) G_s(u) W_{\kappa, \xi_s}(\de u) \de s
    \\&=\llangle H, G(\Phi(\xi_s)+\bar\varphi^\kappa)\rrangle+\frac{1}{\kappa\lambda}\int_0^TH_s(0)G_s(0)(\Phi(\xi_s(0))+\bar\varphi^\kappa(0))\de s
    \\&\phantom{=\llangle H, G(\Phi(\xi_s)+\bar\varphi^\kappa)\rrangle}+\frac{1}{\kappa\delta} \int_0^T H_s(1)G_s(1)(\Phi(\xi_s(1))+\bar\varphi^\kappa(1))\de s,
\end{align*}
and let $L^2_{\kappa, \xi}([0, T]\times[0, 1])$ the Hilbert space of measurable functions $H:[0, T]\times[0, 1]\to\mathbb{R}$ such that $\llangle H, H\rrangle_{\kappa, \xi}<\infty$. 

\begin{definition}[Energy Functional] Given $\kappa>0$ and a function $\xi:[0, T]\times[0, 1]\to[0, \infty)$ such that $\Phi(\xi)\in L^2([0, T]\times[0, 1])$,   let $\mathscr{T}_{\kappa, \xi}$ denote the functional acting on functions $H\in C^{0, 1}([0, T]\times[0, 1])$ via
\begin{equation}\label{eq:def_T}
    \mathscr{T}_{\kappa, \xi}(H):=\llangle \Phi(\xi)-\bar\varphi^\kappa, \partial_u H\rrangle,
\end{equation}
where $\bar\varphi^\kappa_t(u):=\bar\varphi^\kappa(u)$. For $c>0$, we define the energy functional acting on functions $\xi:[0, T]\times[0, 1]\to[0, \infty)$ such that $\Phi(\xi)\in L^2([0, T]\times[0, 1])$ via
\begin{equation}\label{eq:def_energy}
    \energy_{\kappa, c}(\xi):=\mysup_{H\in C^{0, 1}(
    [0, T]\times[0,1])}\left\{\mathscr{T}_{\kappa, \xi}(H)-c \llangle H, H\rrangle_{\kappa, \xi} \right\}.
\end{equation}
\end{definition}

Our first main result provides an energy bound for the family
$\{\rho^\kappa\}_{\kappa>0}$ of weak solutions to
\eqref{eq:zre_robin}, with an upper bound independent of $\kappa$.
This estimate plays a central role in our analysis: together with
the properties of the energy functional established below, it
provides the control needed to study the boundary behaviour of
these solutions and to pass to the limits $\kappa\to0$ and
$\kappa\to\infty$.

Throughout the statements of the following results,
we assume that the initial profile
$\gamma:[0,1]\to[0,\infty)$ is measurable and satisfies
\begin{equation}\label{eq:initial_condition}
0\leq\gamma(u)\leq
\Phi^{-1}\left(
\min\left\{
\frac{\alpha+\beta}{\lambda+\delta},
\frac{\alpha}{\lambda}(1-u)
+\frac{\beta}{\delta}u
\right\}
\right)
\end{equation}
for almost every $u\in[0,1]$.

\begin{theorem}[Energy Estimate]
\label{thm:energy_estimate}
Let $\gamma:[0,1]\to[0,\infty)$ satisfying \eqref{eq:initial_condition}.

Then there exist constants $c>0$ and $M_0<\infty$,
independent of $\kappa$, such that, for every $\kappa>0$,
the unique weak solution $\rho^\kappa$ to \eqref{eq:zre_robin}
with initial profile $\gamma$ satisfies
\begin{equation*}
\energy_{\kappa,c}(\rho^\kappa)\leq M_0.
\end{equation*}
\end{theorem}

Our second main result concerns the regularity of functions
with finite energy. More precisely,
Theorem~\ref{thm:energy_properties} establishes the Sobolev
regularity of $\Phi(\xi)$, provides a representation of its
weak spatial derivative in terms of the Riesz representative
of $\mathscr T_{\kappa,\xi}$, and expresses the energy as a
weighted quadratic norm. Moreover, the generalized boundary
values of the spatial derivative, defined through the Riesz
representation, satisfy the Robin boundary conditions.
These results translate the energy estimate of
Theorem~\ref{thm:energy_estimate} into quantitative control
of the weak spatial derivative of $\Phi(\xi)$ and its
generalized boundary values.

\begin{theorem}[Properties of the Energy Functional]
\label{thm:energy_properties}
Let $\xi\in\mathscr B$, and let
$\Phi:[0,\infty)\to[0,\infty)$ be globally Lipschitz.
Suppose that
$\energy_{\kappa,c}(\xi)\le M_0<\infty$
for some $\kappa,c>0$.

Then there exists a unique
$f\in L^2_{\kappa,\xi}([0,T]\times[0,1])$
such that
\begin{equation*}
\mathscr T_{\kappa,\xi}(H)
=
\llangle f,H\rrangle_{\kappa,\xi}
\end{equation*}
for every $H\in C^{0,1}([0,T]\times[0,1])$. 
Moreover $\Phi(\xi)\in L^2(0,T;\mathcal H^1)$ with its weak spatial derivative given by
\begin{equation*}
\partial_u\Phi(\xi)(t,u)
=
\partial_u\bar\varphi^\kappa(u)
-f_t(u)\bigl(\Phi(\xi_t(u))
+\bar\varphi^\kappa(u)\bigr)
\end{equation*}
for almost every $(t,u)\in(0,T)\times(0,1)$.
Furthermore,
\begin{equation}\label{eq:energy_id}
\energy_{\kappa,c}(\xi)
=\frac{1}{4c}\Bigg\|\frac{\partial_u\Phi(\xi)-\partial_u\bar\varphi^\kappa}{\Phi(\xi)+\bar\varphi^\kappa}\Bigg\|_{\kappa,\xi}^2.
\end{equation}

In addition, the Robin boundary conditions
\begin{equation}\label{eq:boundary_conditions}
\begin{aligned}
\partial_u\Phi(\xi_t)(0)
&=\kappa\bigl(\lambda\Phi(\xi_t(0))-\alpha\bigr),\\
\partial_u\Phi(\xi_t)(1)
&=\kappa\bigl(\beta-\delta\Phi(\xi_t(1))\bigr)
\end{aligned}
\end{equation}
hold for almost every $t\in(0,T)$, where the boundary
values of $\partial_u\Phi(\xi)$ are defined by
\[
\partial_u\Phi(\xi_t)(i)
:=
\partial_u\bar\varphi^\kappa(i)
-f_t(i)\bigl(\Phi(\xi_t(i))
+\bar\varphi^\kappa(i)\bigr),
\qquad i\in\{0,1\}.
\]
These boundary values are understood in a generalized
sense, as the boundary quantities provided by the
Riesz representation, and do not necessarily coincide
with Sobolev traces. We assume that the boundary values
$\Phi(\xi_t(0))$ and $\Phi(\xi_t(1))$
coincide with the Sobolev traces of $\Phi(\xi_t)$
whenever the latter are well-defined.
\end{theorem}

Our final main result establishes strong convergence of the Robin
solutions to the corresponding Neumann and Dirichlet solutions
as $\kappa\to0$ and $\kappa\to\infty$, respectively.
Its proof also relies on the energy estimate in
Theorem~\ref{thm:energy_estimate}, through the representation
and boundary identities in Theorem~\ref{thm:energy_properties}.
In particular, the passage to the Dirichlet limit requires
sufficient control of the boundary derivatives as
$\kappa\to\infty$ to show that the boundary discrepancies,
expressed through the Robin identities as boundary derivatives
divided by $\kappa$, vanish. The energy estimate is a key
ingredient in obtaining this control.

\begin{theorem}[Convergence of Weak Solutions]\label{thm:convergence}
Let $\gamma:[0,1]\to[0,\infty)$ satisfying \eqref{eq:initial_condition}, and for each $\kappa>0$, let $\rho^\kappa:[0, T]\times[0, 1]\to[0, \infty)$ denote the unique weak solution to \eqref{eq:zre_robin}. Then,
\begin{equation}
    \begin{split}
    &\mylim_{\kappa\to 0}\left\|\rho^\kappa-\rho^0\right\|_2=0 \qquad \text{and} \qquad
    \mylim_{\kappa\to \infty}\left\|\rho^\kappa-\rho^\infty\right\|_2=0,
    \end{split}
\end{equation}
where $\rho^0$ and $\rho^\infty$ denote the unique weak solutions to \eqref{eq:zre_neumann} and \eqref{eq:zre_dirichlet}, respectively.
\end{theorem}

The initial-profile condition
\eqref{eq:initial_condition}, imposed in both
Theorems \ref{thm:energy_estimate} and \ref{thm:convergence}, requires $\gamma$ to lie below
$\inf_{\kappa>0}\bar\rho^\kappa$ almost everywhere, as shown
in Lemma \ref{lemma:inf} and Remark \ref{remark:initial condition}. This restriction is used to construct initial
microscopic measures satisfying the stochastic domination
hypothesis required in the proof of the uniform energy estimate.
In particular, it implies that $\gamma$ is bounded.
The comparison principle in Proposition \ref{prop:comparison} then yields
\[
    \sup_{\kappa>0}
    \|\rho^\kappa\|_{L^\infty([0,T]\times[0,1])}<\infty.
\]
Together with the uniform energy estimate, this bound is used
in the compactness and boundary convergence arguments
underlying the proof of Theorem \ref{thm:convergence}.

\section{Proof of the Energy Estimate}\label{sec:energy}
In this section, we prove Theorem \ref{thm:energy_estimate}. Our strategy is based on the construction of a class of interacting particle systems such that their hydrodynamic equation is \eqref{eq:zre_robin}, namely those studied in \cite{agnr25}. We start by recalling the dynamics of these models.

\subsection{The Model} \label{sec:zrp}
For a positive integer $N$, let $I_N$ denote the discrete interval $\{1,\ldots, N-1\}$. This constitutes the discrete space where a set of indistinguishable particles will be moving around, and we will call it the \textit{bulk}. For each $x \in I_N$, the occupation variable $\eta(x)$ will represent the number of particles at site $x$. The \textit{zero-range process} describes the random evolution of particles with no restriction on the total number of particles per site, and therefore its state space of configurations $\eta$ is the set $\Omega_N := \mathbb{N}^{I_N}$.

The rates of the process are defined through a function $g :\mathbb{N}\to [0, \infty)$ satisfying  $g(0)=0$. This last identity mirrors the fact that if there are no particles at the site $x$, then the rate for a jump is null. Moreover, we assume that $g$ is strictly positive on the set of positive integers and that it has bounded increments  in the sense of \eqref{eq:function_g}. This condition ensures that the process is well-defined; see \cite{lig05} for details. 

We consider the Markov process whose infinitesimal generator $\mathcal{L}^{N, \kappa}$ is given by
\begin{equation}\label{eq:generator}
    \mathcal{L}^{N, \kappa}:=\mathcal{L}^N_0 + \kappa \mathcal{L}^N_b,
\end{equation} 
where $\kappa>0$ and where $\mathcal{L}^N_0$ and $\mathcal{L}^N_b$ represent the infinitesimal generators of the bulk dynamics and the boundary dynamics, respectively. These generators act on maps $f:\Omega_N \to \mathbb{R}$ via
\begin{equation*}
    \mathcal{L}^N_0f(\eta):=\sum_{\substack{x, y\in I_N \\ |x-y|=1}} g(\eta(x)) \left[f(\eta^{x,y})-f(\eta)\right]
\end{equation*} 	
and $\mathcal{L}^N_b:=\mathcal{L}^N_l+\mathcal{L}^N_r$, where 
\begin{equation}\label{eq:boundary_gen}
    \begin{split}
    &\mathcal{L}^N_l f(\eta):=\frac{\alpha}{N}\left[f(\eta^{1+})-f(\eta)\right] +\frac{\lambda g(\eta(1))}{N}\left[f(\eta^{1-})-f(\eta)\right],
    \\&\mathcal{L}^N_r f(\eta):=\frac{\beta}{N}\left[f(\eta^{(N-1)+})-f(\eta)\right]+\frac{\delta g(\eta(N-1))}{N}\left[f(\eta^{(N-1)-})-f(\eta)\right],
    \end{split}
\end{equation}
with
\begin{equation*}
    \eta^{x,y}(z):=\left\{\begin{array}{ll}
    \eta(z)-1, & \mbox{~~if~~} z=x,\\
    \eta(z)+1, &\mbox{~~if~~} z=y,\\
    \eta(z), & \mbox{~~if~~} z\ne x,y
    \end{array}\right.
\end{equation*}
and 
\begin{equation*}
    \eta^{x\pm}(z):=\left\{\begin{array}{ll}
    \eta(z) \pm 1, & \mbox{~~if~~} z=x,\\
    \eta(z), & \mbox{~~if~~} z\ne x.
    \end{array}\right.
\end{equation*}

Above, the parameters $\alpha, \beta, \lambda, \delta>0$ are related to density of the reservoirs. We will denote by $\{\eta_t=\eta_t^{N, \kappa}, t\ge0\}$ the Markov process on $\Omega_N$ with generator $N^2\mathcal{L}^{N, \kappa}$. 

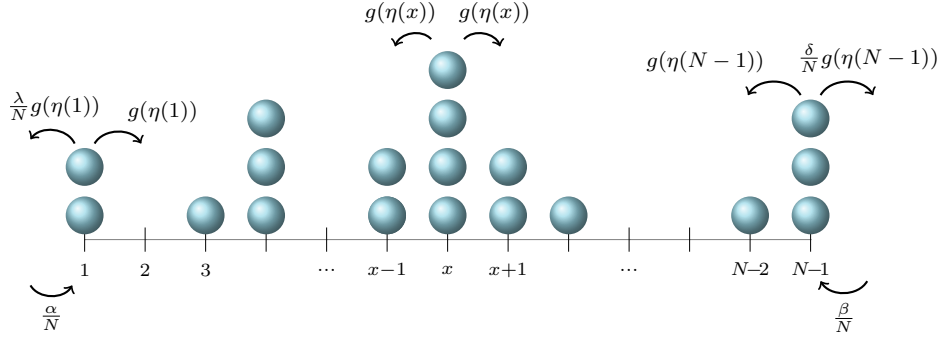
\begin{figure}
\begin{center}
\begin{tikzpicture}[thick, scale=0.8][h!]    
    \draw[step=1cm,gray,very thin] (-6,0) grid (6,0);   
    \foreach \y in {-6,...,6}{
    \draw[color=black,very thin] (\y,-0.2)--(\y,0.2);}    
    \draw (-6,-0.5) node [color=black] {$\scriptstyle{1}$};
    \draw (-5,-0.5) node [color=black] {$\scriptstyle{2}$};
    \draw (-4,-0.5) node [color=black] {$\scriptstyle{3}$};
    \draw (-2,-0.5) node [color=black] {$\scriptstyle{...}$};
    \draw (-1,-0.5) node [color=black] {$\scriptstyle{x-1}$};
    \draw (0,-0.5) node [color=black] {$\scriptstyle{x}$};
    \draw (1,-0.5) node [color=black] {$\scriptstyle{x+1}$};
    \draw (3,-0.5) node [color=black] {$\scriptstyle{...}$};
    \draw (5,-0.5) node [color=black] {$\scriptstyle{N\!-\!2}$};
    \draw (6,-0.5) node [color=black] {$\scriptstyle{N\!-\!1}$};
        
    \node[ball color=ballblue!45!, shape=circle, minimum size=0.5cm] at (-6,1.2) {};
    \node[ball color=ballblue!45!, shape=circle, minimum size=0.5cm]  at (-6.,0.4) {};
    \node[shape=circle,minimum size=0.63cm] (A) at (-6.,2) {};
    \node[shape=circle,minimum size=0.5cm] (C) at (-6.,1.3) {};
    \node[shape=circle,minimum size=0.5cm] (D) at (-4.9,1.3) {};
    \path [->] (C) edge[bend left=60] node[above right] {{\footnotesize{$g(\eta(1))$}}} (D);
    \node[shape=circle,minimum size=0.63cm] (OO) at (-7.1,1.3) {};
    \path [<-] (OO) edge[bend left=60] node[above] {{\footnotesize{$\frac{\lambda}{N}g(\eta(1))$}}} (C);	
    \node[shape=circle,minimum size=0.63cm] (OO0) at (-7.1,-0.4) {};
    \node[shape=circle,minimum size=0.63cm] (C0) at (-6,-0.4) {};		
    \path [<-] (C0) edge[bend left=60] node[below ] {{\footnotesize{$\frac{\alpha}{N}$}}} (OO0);
    
    \node[ball color=ballblue!45!, shape=circle, minimum size=0.5cm]  at (6.,0.4) {};
    \node[ball color=ballblue!45!, shape=circle, minimum size=0.5cm] at (6.,1.2) {};
    \node[ball color=ballblue!45!, shape=circle, minimum size=0.5cm] at (6.,2) {};
    \node[shape=circle,minimum size=0.5cm] (E) at (6.,2.1) {};
    \node[shape=circle,minimum size=0.5cm] (F) at (4.8,2.1) {};
    \path [->] (E) edge[bend right =60] node[above left] {{\footnotesize{$g(\eta(N-1))$}}} (F);
    \node[ball color=ballblue!45!, shape=circle, minimum size=0.5cm]  at (5.,0.4) {};
    \node[shape=circle,minimum size=0.5cm] (H) at (7.2,2.1) {};
    \path [->] (E) edge[bend left=60] node[above] {
    {{\footnotesize{$\qquad\frac{\delta }{N}g(\eta(N-1))$}}}}(H);		
    \node[shape=circle,minimum size=0.5cm] (H0) at (7.1,-0.4) {}; 
    \node[shape=circle,minimum size=0.5cm] (E0) at (6,-0.4) {};		
    \path [->] (H0) edge[bend left=60] node[below ] {{\footnotesize{$\frac{\beta}{N}$}}} (E0);
    
    \node[ball color=ballblue!45!, shape=circle, minimum size=0.5cm] (O) at (-4,0.4) {};
    \node[ball color=ballblue!45!, shape=circle, minimum size=0.5cm]  at (-3.,0.4) {};
    \node[ball color=ballblue!45!, shape=circle, minimum size=0.5cm]  at (-3.,1.2) {};
    \node[ball color=ballblue!45!, shape=circle, minimum size=0.5cm]  at (-3.,2) {};
    \node[ball color=ballblue!45!, shape=circle, minimum size=0.5cm]  at (-1.,0.4) {};
    \node[ball color=ballblue!45!, shape=circle, minimum size=0.5cm]  at (-1.,1.2) {};		
    \node[ball color=ballblue!45!, shape=circle, minimum size=0.5cm]  at (0.,0.4) {};
    \node[ball color=ballblue!45!, shape=circle, minimum size=0.5cm]  at (0.,1.2) {};
    \node[ball color=ballblue!45!, shape=circle, minimum size=0.5cm]  at (0.,2) {};
    \node[ball color=ballblue!45!, shape=circle, minimum size=0.5cm]  at (0.,2.8) {};		
    \node[ball color=ballblue!45!, shape=circle, minimum size=0.5cm]  at (1.,0.4) {};
    \node[ball color=ballblue!45!, shape=circle, minimum size=0.5cm]  at (1.,1.2) {};		
    \node[ball color=ballblue!45!, shape=circle, minimum size=0.5cm]  at (2.,0.4) {};		
    \node[ball color=ballblue!45!, shape=circle, minimum size=0.5cm]  at (-1,0.4) {};
        
    \node[shape=circle,minimum size=0.63cm] (G1) at (0,3) {};
    \node[shape=circle,minimum size=0.63cm] (H1) at (1,2.8) {};
    \path [->] (G1) edge[bend left =60] node[above] {{\footnotesize{$\quad g(\eta(x))$}}} (H1);		
    \node[shape=circle,minimum size=0.63cm] (J1) at (-1,2.8) {};
    \path [<-] (J1) edge[bend left =60] node[above] {{\footnotesize{$g(\eta(x))\quad$}}} (G1);
\end{tikzpicture}
\caption{Microscopic dynamics of the symmetric zero-range process with open boundaries for $\kappa=1$.}\label{fig.1}
\end{center}
\end{figure}

Let $\varphi_N^\kappa:I_N\to[0, \infty)$ be a bounded measurable function. We define the product measure $\bar{\nu}_{\varphi_N^\kappa}$ on $\Omega_N$ with marginals given by 
\begin{equation}\label{eq:prod_measure}
    \bar{\nu}_{\varphi_N^\kappa}\{\eta: \eta(x)=j\}:=\frac{1}{Z(\varphi_N^\kappa(x))}\frac{(\varphi_N^\kappa(x))^j}{g(j)!},
\end{equation}
for $x\in I_N$, with $g(j)!$ and the partition function $Z$ defined in Section \ref{sec:assumptions_Phi}. Denoting by $E_{\nu}[\,\cdot\,]$ the expectation with respect to a measure $\nu$ on $\Omega_N$, a simple computation shows that
\begin{equation*}
    E_{\bar{\nu}_{\varphi_N^\kappa}}[\eta(x)] = \frac{1}{Z(\varphi_N^\kappa(x))}\sum_{j=0}^\infty j\frac{(\varphi_N^\kappa(x))^j}{g(j)!}.
\end{equation*}
Recalling the definition of $R:[0,\varphi^*)\to[0, \infty)$ in \eqref{eq:R}, we can rewrite the expected value of the occupation variable at site $x$ under $\bar{\nu}_{\varphi_N^\kappa}$ as
\begin{equation}\label{eq:expected_R}
    E_{\bar\nu_{\varphi_N^\kappa}}[\eta(x)] = R(\varphi_N^\kappa(x)), \quad x\in I_N.
\end{equation}

The next result can be found in \cite[Equation 20]{lms05}.
\begin{lemma}\label{lemma:inv_measure}
For $\alpha, \lambda, \beta$ and  $\delta$  satisfying 
\begin{equation}\label{eq:params_condition}
    {\mymax\left\{\frac{\alpha}{\lambda}, \frac{\beta}{\delta}\right\}<\varphi^*,}
\end{equation}
and $\kappa>0$ and $N\geq 2$,
the measure $\bar\nu_{\bar{\varphi}_N^\kappa}$, which will be denoted by $\bar{\nu}_{N}^\kappa$, defined in \eqref{eq:prod_measure} with fugacity profile
\begin{equation}\label{eq:fugacity}
    \bar \varphi_N^\kappa(x):=\frac{-\kappa(\alpha \delta - \beta \lambda)(x-1)+\kappa\alpha \delta (N-2) +(\alpha+\beta)N}{\kappa\lambda \delta (N-2)+(\lambda+ \delta)N}, \quad x \in I_N,
\end{equation}
is the unique invariant probability for the Markov process on $\Omega_N$ with infinitesimal generator $\mathcal{L}^{N, \kappa}$, defined in \eqref{eq:generator}. 
\end{lemma}

\begin{remark}\label{remark:barphi_bounds}
For every $\kappa>0$, $N\ge2$, and $x\in I_N$, the stationary
fugacity profile satisfies
\begin{equation*}
    \mymin\left\{\frac{\alpha}{\lambda},
                 \frac{\beta}{\delta}\right\}
    \le \bar\varphi_N^\kappa(x)
    \le
    \mymax\left\{\frac{\alpha}{\lambda},
                 \frac{\beta}{\delta}\right\}.
\end{equation*}
Consequently, the profiles $\bar\varphi_N^\kappa$ are uniformly
bounded above and away from zero, with bounds independent of
$\kappa$ and $N$.
\end{remark}

\begin{remark}\label{remark:limit_fug}
Let $\alpha,\lambda,\beta,\delta>0$. The discrete fugacity
profiles $\bar\varphi_N^\kappa$ approximate the corresponding
macroscopic profiles $\bar\varphi^\kappa$, defined in
\eqref{eq:limit_fug}, with an error of order $N^{-1}$,
uniformly in $\kappa>0$. More precisely, for every $N\ge2$,
\[
    \mysup_{\kappa>0}\mysup_{x\in I_N}
    \left|
    \bar\varphi^\kappa\!\left(\frac{x}{N}\right)
    -\bar\varphi_N^\kappa(x)
    \right|
    \le
    \frac{1}{N}
    \left|\frac{\alpha}{\lambda}-\frac{\beta}{\delta}\right|.
\]
Indeed, a direct computation gives
\[
    \left|
    \bar\varphi^\kappa\!\left(\frac{x}{N}\right)
    -\bar\varphi_N^\kappa(x)
    \right|
    =
    \frac{
    \kappa|\alpha\delta-\beta\lambda|\,
    \left|\delta-\lambda+
    \kappa\lambda\delta\left(\frac{2x}{N}-1\right)\right|
    }{
    (\kappa\lambda\delta+\lambda+\delta)
    \bigl[\kappa\lambda\delta(N-2)+(\lambda+\delta)N\bigr]
    }.
\]
Since $x\in I_N$, we have $|2x/N-1|\le (N-2)/N$, so that
\[
    \left|\delta-\lambda+
    \kappa\lambda\delta\left(\frac{2x}{N}-1\right)\right|
    \le
    \frac{\kappa\lambda\delta(N-2)+(\lambda+\delta)N}{N}.
\]
Consequently,
\[
    \left|
    \bar\varphi^\kappa\!\left(\frac{x}{N}\right)
    -\bar\varphi_N^\kappa(x)
    \right|
    \le
    \frac{\kappa|\alpha\delta-\beta\lambda|}
    {N(\kappa\lambda\delta+\lambda+\delta)}
    \le
    \frac{1}{N}
    \left|\frac{\alpha}{\lambda}-\frac{\beta}{\delta}\right|,
\]
which proves the claimed uniform bound.
\end{remark}

Throughout, we assume that $g$ satisfies \textsc{Assumption (ND)} and \textsc{Assumption (SG)}, given in Section \ref{sec:assumptions_Phi}. These conditions are required as technical assumptions for the results of \cite{agnr25}, which we will need in what follows.

\begin{remark}\label{remark:spectral_gap}
It is worth noting that \textsc{Assumption (SG)} is only used in the proofs of the one-block and two-block estimates in \cite{agnr25}. In fact, we do not need a lower bound of order $\big(\ell(1+\frac{j}{\ell})\big)^{-2}$ for the spectral gap on a box of size $\ell$ with $j$ particles, as long as, for $j\lesssim \ell$, we have \textit{any} positive lower bound $c(\ell)$ depending only on the rate $g$ and the size of the box, decaying at most polynomially in $\ell$, so that $\ell^3/c(\ell)=o(N^2)$ uniformly in $j\lesssim\ell$. 
\end{remark}

We highlight that in particular, under \textsc{Assumption (ND)}, the boundary-driven zero-range process is \textit{attractive}; see \cite[Lemma 2.6]{agnr25}. This means that the dynamics are monotone in the sense that the semigroup 
$\{S(t)\}_{t\ge 0}$ preserves the stochastic order. More precisely, for any 
probability measures $\mu_1 \le \mu_2$, we have $\mu_1 S(t) \le \mu_2 S(t)$ for all $t\ge 0$, where $(S(t)f)(\eta) := \mathbb{E}_\eta[f(\eta_t)]$, and $\mu S(t)$ denotes 
the distribution at time $t$ of the process started from $\mu$. The partial order is defined as follows. Given configurations $\eta,\xi \in \Omega_N$, 
we write $\eta \le \xi$ if $\eta(x) \le \xi(x)$ for all $x \in I_N$. A function 
$f:\Omega_N \to \mathbb{R}$ is said to be \textit{monotone} if $f(\eta) \le f(\xi)$ 
whenever $\eta \le \xi$. Finally, for two probability measures $\mu_1,\mu_2$, we say that $\mu_1 \le \mu_2$ if $\int f \de \mu_1 \le \int f \de\mu_2$ for all monotone functions $f:\Omega_N \to \mathbb{R}$.

\subsection{Hydrodynamic Limit}
We now recall the main result of \cite{agnr25}. Let $\mathcal{M}_+$ be the space of finite, positive measures on $[0,1]$ endowed with the weak topology. Given a configuration $\eta\in\Omega_N$, we associate to it the empirical measure on $[0,1]$ defined by
\begin{equation*}
    \pi^{N, \kappa}_t(\de u) := \frac{1}{N} \sum_{x\in I_N} \eta_t(x) \delta_{\frac{x}{N}}(\de u),
\end{equation*}
where $\delta_u$ stands for the Dirac measure at $u\in[0,1]$. For a measurable function $H:[0,1]\to\mathbb R$, we denote its integral with respect to the measure $\pi_t^{N, \kappa}$ by
\begin{equation*}
    \langle \pi_t^{N, \kappa}, H\rangle := \frac{1}{N}\sum_{x\in I_N} \eta_t(x) H\Big(\frac xN\Big).
\end{equation*} 
Given a measure $\mu^N$ on $\Omega^N$, we denote by $\prob_{\mu^N}^\kappa$ the probability on the Skorokhod space of càdlàg trajectories $\mathcal D([0,T], \Omega_N)$ corresponding to the jump process $\{\eta_t, t\in[0, T]\}$ with generator $N^2\mathcal{L}^{N, \kappa}$ and initial distribution $\mu^N$. Moreover, let $Q^N_\kappa$ denote the probability on $\mathcal D([0,T], \mathcal M_+)$ defined by $Q^N_ \kappa:=\prob_{\mu^N}^\kappa(\pi^{N, \kappa})^{-1}$.

\begin{definition}\label{def:associated_profile}
A sequence $\{\mu^N\}_N$ of probability measures on $\Omega_N$ is said to be \textit{associated} to a measurable and integrable profile $\gamma:[0,1]\to\mathbb[0, \infty)$ if,  for any $\eps>0$ and any continuous function $H:[0,1]\to\mathbb{R}$, the following limit holds:
\begin{equation}\label{eq:initial_profile}
    \mylim_{N\to\infty}
    \mu^N \left\{\eta\in\Omega_N: \left|\frac{1}{N} \sum_{x=1}^{N-1} H\left(\frac{x}{N}\right)\eta(x)
    - \int_0^1 H(u)\gamma(u)\de u \right| > \eps\right\} = 0.
\end{equation}
\end{definition}

Given two measures $\mu, \nu$ on $\Omega_N$, we denote by $H(\mu\,|\,\nu)$ their relative entropy
\begin{equation*}
    H(\mu\,|\,\nu):=\int_{\Omega_N }\mylog\left(\frac{\mu}{\nu}\right)\de \mu.
\end{equation*}
The result below can be found in \cite[Theorem 2.12]{agnr25}.
\begin{theorem}\label{thm:hydro} {{Assume \eqref{eq:params_condition}.}}
Let $\gamma:[0,1]\to[0, \infty)$ be a  measurable  and integrable function and let $\{\mu^N\}_N$ be a sequence of probability measures on $\Omega_N$ associated to $\gamma$, {{in the sense of Definition \ref{def:associated_profile}}} and satisfying the two conditions 
\begin{equation}\label{eq:initial_assumptions}
    \mysup_{\kappa>0}H(\mu^N|\bar\nu_N^\kappa)\lesssim N \quad \textrm{and} \quad \mu^N\le\bar\nu_N^\kappa.
\end{equation}
Then, for each $\kappa>0$ and every $t\in [0,T]$, the sequence of empirical measures $\{\pi^{N, \kappa}_t\}_N$ converges in probability to the absolutely continuous measure $\rho^\kappa_t(u)\de u$, that is, for any $\eps>0$ and for any $H\in C([0,1])$, we have that
\begin{equation*}
    \mylim_{N\to\infty}\prob_{\mu^N}^\kappa \left\{\left|\langle \pi^{N, \kappa}_{t},H\rangle-\int_0^1\rho_t^\kappa(u)H(u)\de u \right|>\eps\right\} =0,
\end{equation*}
where $\rho_t^\kappa(u)$ is the unique weak solution of \eqref{eq:zre_robin} with initial condition $\gamma$.   
\end{theorem}

{{In fact, the result we prove is the following proposition, which implies the result stated above.}}
\begin{proposition}\label{corol:hydro} Under the hypotheses of Theorem \ref{thm:hydro}, for each $\kappa>0$, the sequence of measures $\{Q_\kappa^N\}_N$ converges weakly to $Q_\kappa^*$ as $N\to\infty$, where $Q_\kappa^*$ is a Dirac measure on the path $\pi_t^\kappa(u)=\rho_t^\kappa(u)\de u$ with $\rho^\kappa$ unique weak solution to \eqref{eq:zre_robin}.
\end{proposition}

\begin{lemma}\label{lemma:inf}
For every $u\in[0,1]$, we have
\[
\inf_{\kappa>0}\bar\rho^\kappa(u)
=
\Phi^{-1}\left(
\min\left\{
\frac{\alpha+\beta}{\lambda+\delta},
\frac{\alpha}{\lambda}(1-u)
+\frac{\beta}{\delta}u
\right\}
\right).
\]
\end{lemma}

The proof is given at the end of the appendix.

\begin{remark}\label{remark:initial condition}
Following \cite[Remark 2.13]{agnr25}, the stochastic
domination assumption in Theorem~\ref{thm:hydro}
requires the initial profile $\gamma$ to satisfy
\[
\gamma(u)\leq
\inf_{\kappa>0}\bar\rho^\kappa(u),
\]
for almost every $u\in[0,1]$,
where $\bar\rho^\kappa$ is defined in
Remark~\ref{remark:bar_rho_kappa}.

A natural choice of initial measures satisfying both
\eqref{eq:initial_assumptions} and
\eqref{eq:initial_profile} is the sequence
$\{\nu^N_{\gamma(\cdot)}\}_N$ of product measures
with slowly varying parameters associated with
an initial profile $\gamma$ satisfying the above
inequality; see
\cite[Definition 4.2 and Proposition 4.1]{fmn21}
for details.

By Lemma~\ref{lemma:inf}, this condition is
equivalent to \eqref{eq:initial_condition}.
Therefore, the assumption
\eqref{eq:initial_condition} ensures that the
initial profile satisfies the necessary domination
constraint for applying Theorem~\ref{thm:hydro}
under the stochastic domination hypothesis.
\end{remark}

\subsection{Moment Bounds}
In this section, we prove that the family $\{\rho^\kappa\}_\kappa$ of limit densities of $\{Q_\kappa^N\}_N$ is bounded in $L^p([0, T]\times[0, 1])$ for any $p<\infty$. We highlight that, in fact, under the assumption that the initial condition $\gamma$ appearing in Theorem \ref{thm:hydro} is bounded, we are able to prove that $\{\rho^\kappa\}_\kappa$ is bounded in $L^\infty([0, T]\times[0, 1])$; see Proposition \ref{prop:comparison}. However, while the proof of the $L^\infty$ bound is analytical and requires the knowledge of the hydrodynamic equation, that of the $L^p$ bound can be derived purely from the interacting particle system. We will see that, in particular, as long as one has the $L^p$ bound for any $p>2$, this is enough to derive most of the results of the present work; see Remark \ref{remark:necessary_Linfty}.

The result below is a stronger version of \cite[Lemma 3.6]{agnr25}, and yields a uniform bound along $\kappa>0$ for the $\ell$-moments of the occupation variables under the measure $\bar\nu_N^\kappa$.
\begin{lemma}\label{lemma:bounded_moments}
Let $\bar{\nu}_N^\kappa$ denote the invariant measure from Lemma \ref{lemma:inv_measure}: for each positive integer $\ell$, we have that
\begin{equation*}
    \mysup_{\kappa>0}\mysup_{N}\mysup_{x\in I_N}E_{\bar{\nu}_N^\kappa}[\eta(x)^\ell]<\infty.
\end{equation*}
\end{lemma}

\begin{proof}  
It is not hard to check that, for all $x\in I_N$, we have that $E_{\bar\nu^\kappa_N}[\eta(x)^\ell]=R_\ell(\bar\varphi^\kappa_N(x))$, where $R_\ell:[0, \varphi^*)\to[0, \infty)$ is given by
\begin{equation*}
    R_\ell(\varphi):=\frac{1}{Z(\varphi)}\sum_{j=0}^\infty j^\ell\frac{\varphi^j}{g(j)!},
\end{equation*}
with $\varphi^*$ being the radius of convergence of $Z$. Let $\varphi^{\kappa *}:=\mysup_{N}\mysup_{x\in I_N}\bar\varphi^\kappa_N(x)$, where $\bar\varphi_N^\kappa(x)$ was defined in \eqref{eq:fugacity}: by \eqref{eq:params_condition}, we have that $\varphi^{\kappa *}<\varphi^*$, and moreover, by Remark \ref{remark:barphi_bounds}, we have that
\begin{equation*}
    \mysup_{\kappa>0}\varphi^{\kappa *} = \mymax \left\{\frac{\alpha}{\lambda}, \frac{\beta}{\delta}\right\} =:\varphi^{**}.
\end{equation*}
Therefore, following \cite[Lemma 3.6]{agnr25}, we immediately get
\begin{equation*}
    \mysup_{\kappa>0}\mysup_{N}\mysup_{x\in I_N}E_{\bar{\nu}^\kappa_N}[\eta(x)^\ell]=\mysup_{\kappa>0}\mysup_{N}\mysup_{x\in I_N}R_\ell(\bar\varphi^\kappa_N(x))\le\mysup_{\varphi\in[0, \varphi^{**}]}R_\ell(\varphi)<\infty,
\end{equation*}
as claimed.
\end{proof}

We are now able to prove the $L^p$ bound.

\begin{lemma} \label{lemma:rho_Lp} Under the hypotheses of Theorem \ref{thm:hydro}, the limit point $Q_\kappa^*$ of $\{Q_\kappa^N\}_N$ is concentrated on a path $\pi_t^\kappa(\de u)=\rho_t^\kappa(u)\de u$ such that
\begin{equation*}
    \mysup_{\kappa>0} \int_0^T\int_0^1 \rho_t^\kappa(u)^\ell \de u\de t<\infty
\end{equation*}
for any positive integer $\ell$.
\end{lemma}
\begin{proof}
Let $I_N^\eps:= \{1+\eps N, \ldots, N-1-\eps N \}$, and define 
\begin{align}
    &\bar\eta^{\eps N}(x):=\frac{1}{2\eps N+1}\sum_{|y-x|\le \eps N}\eta(y), \qquad &x\in I_N^\eps, \label{eq:bar_eta}
    \\&\vec\eta^{\eps N}(x):=\frac{1}{\eps N}\sum_{y=x+1}^{x+\eps N}\eta(y), \qquad &x\in\{1, \ldots, \eps N\}, \nonumber
    \\ &\cev\eta^{\eps N}(x):=\frac{1}{\eps N}\sum_{y=x-\eps N}^{x-1}\eta(y), \qquad &x\in\{N-\eps N, \ldots, N-1\}, \nonumber
\end{align}
where $\eps N$ is to be understood as $\lfloor \eps N\rfloor$. Note that the map 
\begin{equation*}
    \eta\mapsto \sum_{x=1}^{\eps N} \big(\vec\eta^{\eps N}(x)\big)^\ell+\sum_{x\in I_N^\eps}\big(\bar\eta^{\eps N}(x)\big)^\ell+\sum_{x=N-\eps N}^{N-1}\big(\cev\eta^{\eps N}(x)\big)^\ell=:F_\ell(\eps N, \eta)
\end{equation*}
is nondecreasing: hence, by \textsc{Assumption (ND)}, the stochastic domination hypothesis in \eqref{eq:initial_assumptions} and attractiveness, we have that
\begin{align*}
    \expected_{\mu^N}\left[\int_0^T \frac{1}{N} F_\ell(\eps N, \eta_s) \de s \right]&\le \expected_{\bar\nu_N^\kappa}\left[\int_0^T \frac{1}{N} F_\ell(\eps N, \eta_s) \de s \right]
    = T E_{\bar\nu_N^\kappa}\left[\frac{1}{N} F_\ell(\eps N, \eta)\right] 
    \\&\lesssim T\mysup_{x\in I_N} E_{\bar\nu_N^\kappa}\left[\eta(x)^\ell\right].
\end{align*}
By Lemma \ref{lemma:bounded_moments}, this implies that
\begin{equation*}
    \mysup_{\kappa>0}\mylimsup_{\eps\to 0}\mylimsup_{N\to\infty}\expected_{\mu^N}\left[\int_0^T \frac{1}{N} F_\ell (\eps N, \eta_s) \de s \right]<\infty.
\end{equation*}
But now, define the function $\vec i_\eps$ on $[0, 1]$ as a continuous approximation of the map $\frac{1}{\eps}\boldsymbol{1}_{(0, \eps]}$ in such a way that $\mylim_{\eps\to0}\|\frac{1}{\eps}\boldsymbol{1}_{(0, \eps]}-\vec i_\eps\|_1=0$. Then, we can write $\vec\eta_s^{\eps N}(x)=\langle\pi_s^N, \vec i_\eps(\cdot-\frac{x}{N})\rangle$ up to an error which vanishes in $L^1$ as $\eps\to0$. By writing $\bar\eta_s^{\eps N}(x)$ and $\cev\eta_s^{\eps N}(x)$ in a similar way and applying Lebesgue's differentiation theorem (which is allowed because, by the argument given in the proof of \cite[Lemma 3.5]{agnr25}, $\rho^\kappa$ is integrable), we deduce that
\begin{equation*}
    \mysup_{\kappa>0}\expected_{Q_\kappa^*}\left[\int_0^T\int_0^1 \rho_t^\kappa(u)^\ell\de u\de t\right]<\infty.
\end{equation*}
{{Proposition}} \ref{corol:hydro} then allows us to remove the expectation.
\end{proof}

\subsection{Proof of Theorem \ref{thm:energy_estimate}}
The proof of Theorem \ref{thm:energy_estimate} is a simple consequence of the proposition below: indeed, to remove the expectation, it suffices to apply {{Proposition}} \ref{corol:hydro}. The outline of our proof is based on that of \cite[Proposition 5.4]{dgn21}; however, the argument requires substantial modifications, which come from the different definition of the energy functional and from the fact that the occupation variables of our underlying particle systems are unbounded.

\begin{proposition} There exist constants $M_0$ and $c$ such that, for any $\kappa>0$, 
\begin{equation*}
    \expected_{Q_\kappa^*}\left[\energy_{\kappa, c}(\rho^\kappa)\right]\le M_0<\infty,
\end{equation*}
where $\rho^\kappa$ is the unique weak solution to \eqref{eq:zre_robin} and $Q_\kappa^*$ is the unique limit point of $\{Q_\kappa^N\}_N$.
\end{proposition}
\begin{proof} Following the same argument given before \cite[Equation (5.16)]{dgn21}, we see that it suffices to show that
\begin{equation*}
    \expected_{Q_\kappa^*}\left[\mymax_{0\le l\le q} \left\{\mathscr{T}_{\kappa, \rho^\kappa}(H^l)-c\llangle H^l, H^l\rrangle_{\kappa, \rho^\kappa} \right\}\right]\le M_0
\end{equation*}
for any $q\in\mathbb{N}$ and for some constant $M_0<\infty$ independent of $q$ and $\kappa$. Note that we can write
\begin{equation}\label{eq:T-c}
    \begin{split}
    \mathscr{T}_{\kappa, \rho^\kappa}(H^l)-c\llangle H^l, H^l\rrangle_{\kappa, \rho^\kappa}=\int_0^T
    \Bigg\{&\int_0^1\{\Phi(\rho_s^\kappa(u))-\bar\varphi^\kappa(u)\}\partial_u H^l_s(u)\de u 
    \\&-c\int_0^1 \{\Phi(\rho_s^\kappa(u))+\bar\varphi^\kappa(u)\}H_s^l(u)^2\de u
    \\&-\frac{c}{\kappa\lambda} \left\{\Phi(\rho^\kappa_s(0))+\bar\varphi^\kappa(0)\right\} H_s^l(0)^2
    \\&-\frac{c}{\kappa\delta} \{\Phi(\rho^\kappa_s(1)) + \bar\varphi^\kappa(1)\}H_s^l(1)^2\Bigg\}\de s. 
    \end{split}
\end{equation}

\textbf{Reduction to the Microscopic Problem.} Given $\eps>0$, for fixed $u\in[0, 1]$, we define three approximations of the identity on $[0, 1]$, namely $\vec\iota_\eps^u:=\frac{1}{\eps}\boldsymbol{1}_{[u, u+\eps)}, \bar\iota_\eps^u:=\frac{1}{2\eps}\boldsymbol{1}_{(u-\eps, u+\eps)}$ and $\cev\iota_\eps^u:=\frac{1}{\eps}\boldsymbol{1}_{(u-\eps, u]}$. We will use the notation
\begin{align*}
    &\langle \rho_s^\kappa, \vec\iota_\eps^u\rangle := \frac{1}{\eps}\int_u^{u+\eps} \rho_s^\kappa(v)\de v, \quad \langle \rho_s^\kappa, \bar\iota_\eps^u\rangle := \frac{1}{2\eps}\int_{u-\eps}^{u+\eps} \rho_s^\kappa(v)\de v  \\&\text{and} \quad \langle \rho_s^\kappa, \cev\iota_\eps^u\rangle := \frac{1}{\eps}\int_{u-\eps}^u \rho_s^\kappa(v)\de v.
\end{align*}
Now, the argument given in the proof of \cite[Lemma 3.5]{agnr25} implies, in particular, that $\rho^\kappa\in L^1([0, T]\times[0, 1])$: hence, we can apply Lebesgue's differentiation theorem, which yields
\begin{equation*}
    \mylim_{\eps\to0}\left|\int_0^T\int_0^1 \Phi(\rho_s^\kappa(u))G_s(u)\de u\de s -\int_0^T \int_\eps^{1-\eps}\Phi\big(\langle \rho_s^\kappa, \bar\iota_\eps^u \rangle\big) G_s(u) \de u\de s\right|=0,
\end{equation*}

The boundary limits follow from the Sobolev regularity
in the definition of a weak solution. Indeed,
$\Phi(\rho^\kappa)\in L^2(0,T;H^1(0,1))$, so that,
for almost every $s\in(0,T)$, the function
$\Phi(\rho_s^\kappa)$ admits a continuous representative
on $[0,1]$, whose endpoint values are its Sobolev traces.

Moreover, by Proposition \ref{prop:comparison}, there exists $c_0<\infty$
such that $0\le\rho^\kappa\le c_0$ almost everywhere.
Since $\Phi^{-1}$ is continuous on
$[\Phi(0),\Phi(c_0)]$, composing the continuous
representative of $\Phi(\rho_s^\kappa)$ with
$\Phi^{-1}$ gives a continuous representative of
$\rho_s^\kappa$. Consequently, for almost every
$s\in(0,T)$,
\[
\frac1\varepsilon\int_0^\varepsilon
\rho_s^\kappa(u)\,du
\longrightarrow \rho_s^\kappa(0),
\qquad
\frac1\varepsilon\int_{1-\varepsilon}^1
\rho_s^\kappa(u)\,du
\longrightarrow \rho_s^\kappa(1),
\]
as $\varepsilon\downarrow0$, where the boundary values
are understood through these representatives.

By the continuity of $\Phi$, the same convergence holds
after applying $\Phi$. Since all the averages above
belong to $[0,c_0]$, the dominated convergence theorem
yields
\[
\lim_{\varepsilon\downarrow0}
\int_0^T
\left|
\Phi\left(\frac1\varepsilon\int_0^\varepsilon
\rho_s^\kappa(u)\,du\right)
-\Phi(\rho_s^\kappa(0))
\right|\,ds=0
\]
and
\[
\lim_{\varepsilon\downarrow0}
\int_0^T
\left|
\Phi\left(\frac1\varepsilon\int_{1-\varepsilon}^1
\rho_s^\kappa(u)\,du\right)
-\Phi(\rho_s^\kappa(1))
\right|\,ds=0.
\]
Multiplying the respective differences by $G_s(0)$
and $G_s(1)$, with
$G\in C([0,T]\times[0,1])$, preserves these limits
because $G$ is bounded.

Thus, we only need to show that
\begin{align*}
    \mylimsup_{\eps\to 0} \expected_{Q_\kappa^*}\Bigg[\mymax_{0\le l\le q} \Bigg\{& \int_0^T \Bigg(\int_\eps^{1-\eps}\{\Phi\big(\langle \rho_s^\kappa, \bar\iota_\eps^u \rangle\big) -\bar\varphi^\kappa(u)\}\partial_u H_s^l(u) \de u 
    \\&-c\int_\eps^{1-\eps} \{\Phi\big(\langle \rho_s^\kappa, \bar\iota_\eps^u \rangle\big)+\bar\varphi^\kappa(u)\}H_s^l(u)^2\de u
    \\&-\frac{c}{\kappa\lambda} \{\Phi\big(\langle \rho_s^\kappa, \vec\iota_\eps^0 \rangle\big) + \bar\varphi^\kappa(0)\}H_s^l(0)^2
    \\&-\frac{c}{\kappa\delta} \{\Phi\big(\langle \rho_s^\kappa, \cev\iota_\eps^1 \rangle\big) +\bar\varphi^\kappa(1)\}H_s^l(1)^2\Bigg)\de s\Bigg\}\Bigg]\le M_0.
\end{align*}
But now, by a standard semi-continuity argument (see, for example, \cite[Equation (5.20) and below]{dgn21}), since $Q_\kappa^*$ is the limit point of $\{Q_\kappa^N\}_N$ and using Portmanteau's theorem, we can bound the expectation in the previous display from above by
\begin{align*}
    \myliminf_{N\to \infty} \expected_{Q_\kappa^N}\Bigg[\mymax_{0\le l\le q} \Bigg\{& \int_0^T \Bigg(\int_\eps^{1-\eps}\{\Phi\big(\langle \pi^N_s, \bar\iota_\eps^u \rangle\big) -\bar\varphi^\kappa(u)\}\partial_u H_s^l(u) \de u 
    \\&-c\int_\eps^{1-\eps} \{\Phi\big(\langle  \pi^N_s, \bar\iota_\eps^u \rangle\big)+\bar\varphi^\kappa(u)\}H_s^l(u)^2\de u
    \\&-\frac{c}{\kappa\lambda} \{\Phi\big(\langle  \pi^N_s, \vec\iota_\eps^0 \rangle\big) + \bar\varphi^\kappa(0)\}H_s^l(0)^2
    \\&-\frac{c}{\kappa\delta} \{\Phi\big(\langle \pi^N_s, \cev\iota_\eps^1 \rangle\big) +\bar\varphi^\kappa(1)\}H_s^l(1)^2\Bigg)\de s\Bigg\}\Bigg]
\end{align*}
Recall the definitions of $I_N^\eps$ and of the block averages
$\bar\eta^{\eps N}$, $\vec\eta^{\eps N}$, and $\cev\eta^{\eps N}$
in \eqref{eq:bar_eta} and the subsequent discussion.
We approximate the spatial integrals by Riemann sums and the
empirical-measure averages by the corresponding discrete block
averages. The errors in these approximations are understood in
the iterated sense
\[
    \limsup_{\eps\to0}\limsup_{N\to\infty}
    \expected_{Q_\kappa^N}\bigl[|R_{N,\eps}|\bigr]=0,
\]
where $R_{N,\eps}$ denotes the total approximation error, for
each fixed $\kappa>0$. Since only finitely many test functions
are involved, this control also allows us to pass between the
corresponding maxima.

Moreover, Remark~\ref{remark:limit_fug} allows us to replace
$\bar\varphi^\kappa(x/N)$ by $\bar\varphi_N^\kappa(x)$ in the
bulk terms. At the boundary, the same estimate, together with
the explicit affine form of $\bar\varphi^\kappa$, gives
\[
    \bigl|\bar\varphi^\kappa(0)-\bar\varphi_N^\kappa(1)\bigr|
    +
    \bigl|\bar\varphi^\kappa(1)
          -\bar\varphi_N^\kappa(N-1)\bigr|
    \le \frac{C}{N},
\]
with $C$ independent of $\kappa$ and $N$.
The resulting errors in the boundary terms therefore vanish
as $N\to\infty$ for each fixed $\kappa>0$.
It thus remains to bound
\begin{equation}\label{eq:pre_repl_limit}
    \begin{split}\mylimsup_{\eps\to0}\myliminf_{N\to\infty} \expected_{Q_\kappa^N} \Bigg[\mymax_{0\le l\le q}\Bigg\{&\int_0^T\Bigg(\frac{1}{N}\sum_{x\in I_N^\eps} \left\{\Phi\big(\bar\eta^{\eps N}_s(x)\big) - \bar\varphi^\kappa_N(x)\right\} \partial_u H^l_s\left(\frac{x}{N}\right) 
    \\&-\frac{c}{N} \sum_{x\in I_N^\eps}\left\{\Phi\big(\bar\eta^{\eps N}_s(x)\big) + \bar\varphi_N^\kappa(x)\right\}H_s^l\left(\frac{x}{N}\right)^2
    \\&-\frac{c}{\kappa\lambda} \left\{\Phi\big(\vec\eta^{\eps N}_s(1) \big)+ \bar\varphi_N^\kappa(1)\right\} H_s^l(0)^2
    \\&-\frac{c}{\kappa\delta} \left\{\Phi\big(\cev\eta^{\eps N}_s(N-1)\big)+ \bar\varphi_N^\kappa(N-1)\right\}H_s^l(1)^2\Bigg)\de s\Bigg\}\Bigg].
\end{split}
\end{equation}

\medskip
\textbf{Replacement Lemmas.} Our goal is now to ``replace $\Phi$ by $g$", which we will do via the replacement lemmas of \cite{agnr25}. By the bulk replacement lemma \cite[Lemma 3.9]{agnr25}, we have that, for any $G\in C([0, T]\times[0, 1])$,
\begin{equation*}
    \expected_{Q_\kappa^N} \left[\left|\int_0^T\frac{1}{N}\sum_{x\in I_N^\eps} \Big\{\Phi\big(\bar\eta^{\eps N}_s(x)\big) - \frac{1}{2\eps N+1}\sum_{|y-x|\le \eps N}g(\eta_s(y))\Big\} G_s\left(\frac{x}{N}\right) \right|\right]
\end{equation*}
vanishes as $N\to\infty$ and $\eps\to0$. Moreover, by the argument given below \cite[Equation (3.23)]{agnr25}, we also have that
\begin{equation*}
    \expected_{Q_\kappa^N} \left[\left|\int_0^T\frac{1}{N}\sum_{x\in I_N^\eps} \Big\{ \frac{1}{2\eps N+1}\sum_{|y-x|\le \eps N}g(\eta_s(y))-g(\eta_s(x))\Big\} G_s\left(\frac{x}{N}\right) \right|\right]
\end{equation*}
vanishes as $N\to\infty$ and $\eps\to0$. But then, since $\mymax_{0\le l\le q} |a(H^l)| \le \sum_{l=0}^q |a(H^l)|$, these two limits are enough to see that we can replace $\Phi(\bar\eta_s^{\eps N}(x))$ by $g(\eta_s(x))$ (or, by the same argument, by $g(\eta_s(x+1))$) in the first and second lines of \eqref{eq:pre_repl_limit}. Similarly, by the boundary replacement lemma \cite[Lemma 3.14]{agnr25} and by arguing in a similar way as above, we can also substitute $\Phi(\vec\eta^{\eps N}_s(1))$ and $\Phi(\cev\eta^{\eps N}_s(N-1))$ in the third and fourth line of \eqref{eq:pre_repl_limit}, respectively, with $g(\eta_s(1))$ and $g(\eta_s(N-1))$, respectively. Hence, we are only left to estimate
\begin{equation}\label{eq:limit_repl}
    \begin{split}
    \mylimsup_{\eps\to0}\myliminf_{N\to\infty} \expected_{Q_\kappa^N} \Bigg[\mymax_{0\le l\le q}\Bigg\{&\int_0^T\Bigg(\frac{1}{N}\sum_{x\in I_N^\eps} \left\{g(\eta_s(x))-\bar\varphi_N^\kappa(x)\right\} \partial_u H^l_s\left(\frac{x}{N}\right) 
    \\&-\frac{c}{N}\sum_{x\in I_N^\eps} \left\{\frac{g(\eta_s(x))+g(\eta_s(x+1))}{2}+\bar\varphi_N^\kappa(x)\right\}  H^l_s\left(\frac{x}{N}\right)^2
    \\&-\frac{c}{\kappa\lambda} \left\{g(\eta_s(1))+\bar\varphi_N^\kappa(1)\right\}H_s^l(0)^2
    \\&-\frac{c}{\kappa\delta} \left\{g(\eta_s(N-1)) +\bar\varphi_N^\kappa(N-1)\right\}H_s^l(1)^2\Bigg)\de s\Bigg\}\Bigg].
    \end{split}
\end{equation}
The choice of replacing $\Phi(\bar\eta_s^{\eps N}(x))$ by the average $\frac{g(\eta_s(x))+g(\eta_s(x+1))}{2}$ in the second line will become clear later in the proof.

The idea now is to extend the sum above from $I_N$ to $I_N^\varepsilon$. For that we start by observing  that, by \eqref{eq:function_g}, we have that $g(\eta(x))\le g^* \eta(x)$. But now, the map $\eta\mapsto \eta(x)$ is nondecreasing: hence, by \textsc{Assumption (ND)}, the stochastic domination hypothesis in \eqref{eq:initial_assumptions} and attractiveness, we have that {{the sum in the left boundary,}}
\begin{align}
    &\expected_{Q_\kappa^N} \left[{{\mymax_{0\le l\le q}}}\left|\int_0^T\frac{1}{N}\sum_{x=1}^{\eps N} \left\{g(\eta_s(x))-\bar\varphi_N^\kappa(x)\right\} \partial_u H^l_s\left(\frac{x}{N}\right)\de s\right|\right]\label{eq:eps_bound}
    \\&\le{{\sum_{l=0}^q}}  \Bigg\{g^*\expected_{Q_\kappa^N} \left[\int_0^T\frac{1}{N}\sum_{x=1}^{\eps N} \eta_s(x) \left|\partial_u H^l_s\left(\frac{x}{N}\right)\right|\de s\right] \\&\qquad\qquad\qquad\qquad+ \int_0^T\frac{1}{N}\sum_{x=1}^{\eps N} \bar\varphi_N^\kappa(x)\left|\partial_u H^l_s\left(\frac{x}{N}\right)\right|\de s\Bigg\}\nonumber
    \\& \le \eps g^* T{{\sum_{l=0}^q}} \|\partial_u H^l\|_\infty \mysup_N\mysup_{x\in I_N}E_{\bar\nu_N^\kappa}[\eta(x)]+\eps T{{\sum_{l=0}^q}}  \|\partial_u H^l\|_\infty \mysup_{N}\mysup_{x\in I_N} \bar\varphi_N^\kappa(x),\nonumber
\end{align}
 vanishes as $N\to\infty$ and $\eps\to0$. 
 By arguing in a similar way for the sum at the right boundary and using again $\mymax_{0\le l\le q} a(H^l) \le \sum_{l=0}^q a(H^l)$, we have thus proved that we only need to estimate \eqref{eq:limit_repl} without the limit in $\eps$ and with the sums in the first two lines running on $I_N$ instead of $I_N^\eps$, {{since the argument for the second line is similar}}.
Now, by performing an integration by parts, we can write
\begin{align}
    \frac{1}{N}&\sum_{x=1}^{N-1}\left\{g(\eta_s(x))-\bar\varphi_N^\kappa(x)\right\} \partial_u H^l_s\left(\frac{x}{N}\right) \nonumber
    \\&= \sum_{x=1}^{N-2} \big\{g(\eta_s(x))-g(\eta_s(x+1))\big\}  H^l_s\left(\frac{x}{N}\right) \label{eq:change_ratio}
    \\&\phantom{=}+ \sum_{x=1}^{N-2} \big\{\bar\varphi_N^\kappa(x+1)-\bar\varphi_N^\kappa(x)\big\}  H^l_s\left(\frac{x}{N}\right)\nonumber 
    \\& \phantom{=}+  (\bar\varphi_N^\kappa(1)-g(\eta_s(1))) H^l_s(0)+(g(\eta_s(N-1))-\bar\varphi_N^\kappa(N-1)) H^l_s(1)\nonumber
\end{align}
up to a correction of order $\frac{1}{N}$ (which vanishes by a similar argument to the one used to bound \eqref{eq:eps_bound}). Note that we can rewrite the first sum in \eqref{eq:change_ratio} as
\begin{align*}
    &\sum_{x=1}^{N-2} \left\{g(\eta_s(x))-\frac{\bar\varphi_N^\kappa(x)}{\bar\varphi_N^\kappa(x+1)}g(\eta_s(x+1))\right\}  H^l_s\left(\frac{x}{N}\right)
    \\&\phantom{=}+\sum_{x=1}^{N-2} \left\{\frac{\bar\varphi_N^\kappa(x)}{\bar\varphi_N^\kappa(x+1)}-1\right\}g(\eta_s(x+1))  H^l_s\left(\frac{x}{N}\right).
\end{align*}
The choice of this rewriting will become clear later in the proof; see \eqref{eq:zero_bulk} and below. Now, by \cite[Lemma 3.3]{agnr25} and Remark \ref{remark:barphi_bounds}, we have that
\begin{equation*}
    \frac{E_{\bar\nu_N^\kappa}[g(\eta(x))^2]}{\bar\varphi_N^\kappa(x)}\le g^*+\bar\varphi_N^\kappa(x)\le g^*+\mymax\left\{\frac{\alpha}{\lambda}, \frac{\beta}{\delta}\right\}
\end{equation*}
for all $x\in I_N$. Moreover, note that
\begin{equation}\label{eq:barphi_diff}
    \bar\varphi_N^\kappa(x+1) -\bar\varphi_N^\kappa(x)\equiv \frac{\kappa\beta \lambda-\kappa\alpha \delta}{\kappa\lambda \delta (N-2)+(\lambda+ \delta)N}
\end{equation}
for all $x\in I_N$, and hence, by Remark \ref{remark:barphi_bounds},
\begin{equation*}
    \left\{\frac{\bar\varphi_N^\kappa(x)}{\bar\varphi_N^\kappa(x+1)}-1\right\}^2 \le \frac{\left(\frac{\beta\lambda-\alpha\delta}{\lambda\delta(N-2)}\right)^2}{\mymin\left\{\frac{\alpha}{\lambda}, \frac{\beta}{\delta}\right\}^2}.
\end{equation*}

By Young's inequality, for every trajectory and every
$G\in C([0,T]\times[0,1])$,
 we have that

\begin{align*}
    &\int_0^T\sum_{x=1}^{N-2} \left\{\frac{\bar\varphi_N^\kappa(x)}{\bar\varphi_N^\kappa(x+1)}-1\right\}g(\eta_s(x+1))  G_s\left(\frac{x}{N}\right)\de s
    \\&-\int_0^T \frac{1}{N}\sum_{x=1}^{N-1}\bar\varphi_N^\kappa(x)G_s\left(\frac{x}{N}\right)^2\de s
    \\&\le \int_0^T \frac{N}{4}\sum_{x=1}^{N-2} \left\{\frac{\bar\varphi_N^\kappa(x)}{\bar\varphi_N^\kappa(x+1)}-1\right\}^2 \frac{g(\eta_s(x+1)) ^2}{\bar\varphi_N^\kappa(x)}\de s =:B_N^\kappa.
\end{align*}

The right-hand side is nonnegative and independent of $G$.
Moreover, by stochastic domination and the second-moment
bound,
\begin{align*}
    \expected_{Q_\kappa^N}[B_N^\kappa]
    &\le
    \frac{TN}{4}\sum_{x=1}^{N-2}
    \left\{\frac{\bar\varphi_N^\kappa(x)}{\bar\varphi_N^\kappa(x+1)}-1\right\}^2 \frac{ E_{\bar\nu_N^\kappa}[g(\eta(x+1))^2]}{\bar\varphi_N^\kappa(x)}
    \\    \\& \le \frac{TN}{4(N-2)}\,
    \frac{\left(\frac{\beta\lambda-\alpha\delta}{\lambda\delta}\right)^2}
    {\mymin\left\{\frac{\alpha}{\lambda}, \frac{\beta}{\delta}\right\}^2}
    \left(g^*+\mymax\left\{\frac{\alpha}{\lambda}, \frac{\beta}{\delta}\right\}\right)
    \\&\lesssim \left(\frac{\beta\lambda-\alpha\delta}{\lambda\delta\mymin\left\{\frac{\alpha}{\lambda}, \frac{\beta}{\delta}\right\}}\right)^2\left(g^*+\mymax\left\{\frac{\alpha}{\lambda}, \frac{\beta}{\delta}\right\}\right).
\end{align*}
Above, the constant in $\lesssim$ depends on $T$ but not $\kappa$ nor $N$.
Thus, this contribution is bounded uniformly in $\kappa>0$
and $N\ge3$. Since $B_N^\kappa$ does not depend on the test
function, the same bound applies when taking the maximum
over $H^0,\ldots,H^q$.

Thus, we have proved that we only need to bound
\begin{equation*}
    \begin{split}      
    \myliminf_{N\to\infty} \expected_{Q_\kappa^N} \Bigg[\mymax_{0\le l\le q}\Bigg\{&\int_0^T\Bigg(\sum_{x=1}^{N-2} \left\{g(\eta_s(x))-\frac{\bar\varphi_N^\kappa(x)}{\bar\varphi_N^\kappa(x+1)}g(\eta_s(x+1))\right\}  H^l_s\left(\frac{x}{N}\right) 
    \\&+\sum_{x=1}^{N-2} \left\{\bar\varphi_N^\kappa(x+1)-\bar\varphi_N^\kappa(x)\right\}  H^l_s\left(\frac{x}{N}\right) 
    \\&-\frac{c}{N}\sum_{x=1}^{N-2}{{ \frac{g(\eta_s(x))+g(\eta_s(x+1))}{2}}} H^l_s\left(\frac{x}{N}\right)^2 \\&-\frac{c-1}{N}\sum_{x=1}^{N-1} \bar\varphi_N^\kappa(x)H^l_s\left(\frac{x}{N}\right)^2  
    \\&+\left\{\bar\varphi_N^\kappa(1)-g(\eta_s(1))\right\} H^l_s(0)
    \\&+\left\{g(\eta_s(N-1))-\bar\varphi_N^\kappa(N-1)\right\} H^l_s(1)
    \\&-\frac{c}{\kappa\lambda} \left\{g(\eta_s(1))+\bar\varphi_N^\kappa(1)\right\}H_s^l(0)^2
    \\&-\frac{c}{\kappa\delta} \left\{g(\eta_s(N-1)) +\bar\varphi_N^\kappa(N-1)\right\}H_s^l(1)^2\Bigg)\de s\Bigg\}\Bigg].
\end{split}
\end{equation*}

\medskip \textbf{Variational Formula.} By the entropy inequality and Jensen's inequality (for a precise argument, see the derivation of \cite[Equation (3.42)]{agnr25}), as well as the argument given below \cite[Equation (5.25)]{dgn21}, it suffices to bound
\begin{align*}
    &\frac{H(\mu_N|\bar\nu_N^\kappa)}{N}
    \\&+\frac{1}{N} \mylog  \expected_{\bar\nu_N^\kappa} \Bigg[\myexp\Bigg\{\int_0^T\Bigg(N\sum_{x=1}^{N-2} \left\{g(\eta(x))-\frac{\bar\varphi_N^\kappa(x)}{\bar\varphi_N^\kappa(x+1)}g(\eta(x+1))\right\}  H_s\left(\frac{x}{N}\right) 
    \\&+N\sum_{x=1}^{N-2} \left\{\bar\varphi_N^\kappa(x+1)-\bar\varphi_N^\kappa(x)\right\}H_s\left(\frac{x}{N}\right) 
    \\&-c\sum_{x=1}^{N-2} {{ \frac{g(\eta(x))+g(\eta(x+1))}{2}}} H_s\left(\frac{x}{N}\right)^2-(c-1)\sum_{x=1}^{N-1}  \bar\varphi_N^\kappa(x)H_s\left(\frac{x}{N}\right)^2
    \\&+N\left\{\bar\varphi_N^\kappa(1)-g(\eta(1))\right\} H_s(0)
    \\&+N\left\{g(\eta(N-1))-\bar\varphi_N^\kappa(N-1)\right\} H_s(1)
    \\&-\frac{Nc}{\kappa\lambda} \left\{g(\eta(1))+\bar\varphi_N^\kappa(1)\right\}H_s(0)^2
    \\&-\frac{Nc}{\kappa\delta} \left\{g(\eta(N-1)) +\bar\varphi_N^\kappa(N-1)\right\}H_s(1)^2\Bigg)\de s\Bigg\}\Bigg]
\end{align*}
for some fixed $H\in C^{0,1}([0, T]\times[0,1])$. Now, by the entropy assumption $H(\mu^N|\bar\nu_N^\kappa)\lesssim N$, the entropy term is bounded from above by a constant independent of $\kappa$. As for the remaining term, by the Feynman-Kac formula \cite[Lemma A.1]{bmns17}, it is bounded from above by
\begin{equation}\label{eq:sup_all}
    \begin{split}
   &\int_0^T   \mysup_f\Bigg(\int_{\Omega_N} \sum_{x=1}^{N-2} \left\{g(\eta(x)) - \frac{\bar\varphi^\kappa_N(x)}{\bar\varphi^\kappa_N(x+1)}g(\eta(x+1))\right\} H_s\left(\frac{x}{N}\right) f(\eta) \de\bar\nu_N^\kappa
    \\&+ \sum_{x=1}^{N-2} \left\{\bar\varphi_N^\kappa(x+1) -\bar\varphi_N^\kappa(x)\right\} H_s\left(\frac{x}{N}\right) 
    \\&-\frac{c}{N}\int_{\Omega_N} \sum_{x=1}^{N-2} {{ \frac{g(\eta(x))+g(\eta(x+1))}{2}}} H_s\left(\frac{x}{N}\right)^2 f(\eta) \de\bar\nu_N^\kappa\\&-\frac{c-1}{N}\sum_{x=1}^{N-1} \bar\varphi_N^\kappa(x) H_s\left(\frac{x}{N}\right)^2\\
    &+\int_{\Omega_N}\Big( \left\{\bar\varphi_N^\kappa(1)-g(\eta(1))\right\} H_s(0)-\frac{c}{\kappa\lambda} \left\{g(\eta(1))+\bar\varphi_N^\kappa(1)\right\}H_s(0)^2\Big)\,f(\eta)\de\bar\nu_N^\kappa
    \\
        \end{split}
\end{equation}
 \begin{equation*}
    \begin{split}  
    &+\int_{\Omega_N}\Big(\left\{g(\eta(N-1))-\bar\varphi_N^\kappa(N-1)\right\} H_s(1)
   \\&\qquad\qquad\qquad -\frac{c}{\kappa\delta} \left\{g(\eta(N-1)) +\bar\varphi_N^\kappa(N-1)\right\}H_s(1)^2\Big)\,f(\eta)\de\bar\nu_N^\kappa
    \\&+N\langle \mathcal{L}^{N, \kappa}\sqrt{f}, \sqrt{f}\rangle_{\bar\nu_N^\kappa} \Bigg)\de s,
    \end{split}
\end{equation*}
where the supremum runs along all probability densities $f$ with respect to $\bar\nu_N^\kappa$. 
Let now
\begin{align*}
    &\mathscr{D}_N^l(\sqrt{f}):=\frac{\kappa\lambda}{2N}\int_{\Omega_N}g(\eta(1))\left[\sqrt{f(\eta^{1-})}-\sqrt{f(\eta)}\right]^2\de \bar\nu_N^\kappa,    \\&\mathscr{D}_N^0(\sqrt{f}):=\frac{1}{2}\int_{\Omega_N}\sum_{\substack{x, y\in I_N \\ |x-y|=1}}g(\eta(x))\left[\sqrt{f(\eta^{x, y})}-\sqrt{f(\eta)}\right]^2 \de \bar\nu_N^\kappa,
    \\&\mathscr{D}_N^r(\sqrt{f}):=\frac{\kappa\delta}{2N}\int_{\Omega_N}g(\eta(N-1))\left[\sqrt{f(\eta^{(N-1)-})}-\sqrt{f(\eta)}\right]^2\de \bar\nu_N^\kappa .
\end{align*}
By \cite[Lemma 3.4]{gnr25}, we have that
\begin{equation}\label{eq:dirichlet_bounds}
    \langle \mathcal{L}^{N, \kappa}\sqrt{f}, \sqrt{f}\rangle_{\bar\nu_N^\kappa}\le -\mathscr{D}_N^l(\sqrt{f})-\mathscr{D}_N^0(\sqrt{f})-\mathscr{D}_N^r(\sqrt{f}).
\end{equation}

\medskip
\textbf{The Bulk Terms.} We start by treating the bulk terms, namely the top three lines of \eqref{eq:sup_all}. We write
\begin{align}
    &\int_{\Omega_N}\sum_{x=1}^{N-2} \left\{g(\eta(x)) - \frac{\bar\varphi_N(x)}{\bar\varphi_N(x+1)}g(\eta(x+1))\right\} H_s\left(\frac{x}{N}\right) f(\eta) \de\bar\nu_N^\kappa\nonumber
    \\&=\int_{\Omega_N}\sum_{x=1}^{N-2} g(\eta(x))\left[f(\eta) -f(\eta^{x, x+1}) \right]H_s\left(\frac{x}{N}\right) \de\bar\nu_N^\kappa
    \label{eq:young_bulk}
    \\&\phantom{=}+\int_{\Omega_N}\sum_{x=1}^{N-2} \left\{g(\eta(x))f(\eta^{x, x+1}) - \frac{\bar\varphi_N(x)}{\bar\varphi_N(x+1)}g(\eta(x+1))f(\eta)\right\} H_s\left(\frac{x}{N}\right) \de\bar\nu_N^\kappa. \label{eq:zero_bulk}
\end{align}
By performing the change of variable $\eta\mapsto\eta^{x+1, x}$ in the leftmost term of \eqref{eq:zero_bulk} and noting that $\frac{\bar\nu_N^\kappa(\eta^{x+1,x})}
{\bar\nu_N^\kappa(\eta)}
=
\frac{\bar\varphi_N^\kappa(x)}
{\bar\varphi_N^\kappa(x+1)}
\frac{g(\eta(x+1))}
{g(\eta(x)+1)}$, it is easy to see that 
\begin{align*}
    \eqref{eq:zero_bulk}=\int_{\Omega_N}\sum_{x=1}^{N-2}\Bigg\{&g(\eta(x)+1)
\frac{\bar\varphi_N^\kappa(x)}
{\bar\varphi_N^\kappa(x+1)}
\frac{g(\eta(x+1))}
{g(\eta(x)+1)}
\\&- \frac{\bar\varphi_N(x)}{\bar\varphi_N(x+1)}g(\eta(x+1))\Bigg\} f(\eta)H_s\left(\frac{x}{N}\right) \de\bar\nu_N^\kappa=0.
\end{align*}
On the other hand, following the proof of \cite[Lemma 5.7.3]{kl99}, we also have that

\begin{equation*}
    \eqref{eq:young_bulk}\le 2\int_{\Omega_N} \frac{1}{N}\sum_{x=1}^{N-2} \frac{g(\eta(x))+g(\eta(x+1))}{2} H_s\left(\frac{x}{N}\right)^2 f(\eta) \de\bar\nu_N^\kappa+N\mathscr{D}^0_{N}\left(\sqrt{f}\right),
\end{equation*}
Recalling \eqref{eq:dirichlet_bounds}, {{for the choice $c=2$}} it only remains to treat the second and fourth lines of \eqref{eq:sup_all}. To this end, recalling \eqref{eq:barphi_diff} and by Young's inequality, we see that
\begin{align*}
    &\sum_{x=1}^{N-2} \left\{\bar\varphi_N^\kappa(x+1) -\bar\varphi_N^\kappa(x)\right\} H_s\left(\frac{x}{N}\right) \\&\le \frac{N}{2}\left(\frac{\kappa\beta \lambda-\kappa\alpha \delta}{\kappa\lambda \delta (N-2)}\right)^2\sum_{x=1}^{N-1}\frac{1}{\bar\varphi_N^\kappa(x)}+\frac{1}{2N}\sum_{x=1}^{N-1} \bar\varphi_N^\kappa(x)H_s\left(\frac{x}{N}\right)^2.
\end{align*}
But then, by Remark \ref{remark:barphi_bounds}, the first summand on the right-hand side of the last display is bounded from above by a constant independent of $\kappa$, whereas the second term and the sum in the fourth line of \eqref{eq:sup_all}  {{are negative for the choice $c=2$}}. This concludes the bound.

\medskip
\textbf{The Boundary Terms.} It only remains to treat the boundary terms, namely the bottom three lines of \eqref{eq:sup_all}. We can write the {{first integral appearing in the fifth line of \eqref{eq:sup_all}}} as
\begin{align}
    &\int_{\Omega_N}  \big\{\bar\varphi_N^\kappa(1)-g(\eta(1)) \}H_s(0) f(\eta)\de\bar\nu_N^\kappa\nonumber
    \\&=H_s(0)\int_{\Omega_N} \left\{\bar\varphi_N^\kappa(1)f(\eta)-g(\eta(1))f(\eta^{1-})\right\}\de\bar\nu_N^\kappa\label{eq:change_1}
    \\&\phantom{=} -H_s(0) \int_{\Omega_N} g(\eta(1))\left[f(\eta)-f(\eta^{1-})\right]\de\bar\nu_N^\kappa.\label{eq:young_1}
\end{align}
Performing the change of variable $\eta\mapsto \eta^{1+}$ in the rightmost term of \eqref{eq:change_1} and using the identity $\frac{\bar\nu_N^\kappa(\eta^{1+})}{\bar\nu_N^\kappa(\eta)}=\frac{\bar\varphi_N^\kappa(1)}{g(\eta(1)+1)}$
(which follows immediately from the definition of $\bar\nu_N^\kappa$), we see that
\begin{align*}
    \eqref{eq:change_1}&=H_s(0)\int_{\Omega_N} \left\{\bar\varphi_N^\kappa(1)-\bar\varphi_N^\kappa(1)\right\} f(\eta) \de\bar\nu_N^\kappa=0.
\end{align*}
On the other hand, by Young's inequality, for any $A>0$ we have that
\begin{align*}
    \eqref{eq:young_1}&\le  \frac{A}{2} \int_{\Omega_N} g(\eta(1))\left[\sqrt{f(\eta^{1-})}-\sqrt{f(\eta)}\right]^2\de\bar\nu_N^\kappa
    \\&+\frac{H_s(0)^2}{2A}\int_{\Omega_N} g(\eta(1))\left[\sqrt{f(\eta^{1-})}+\sqrt{f(\eta)}\right]^2\de\bar\nu_N^\kappa.
\end{align*}
Choosing $A=\kappa\lambda$ and recalling \eqref{eq:dirichlet_bounds}, the first term cancels with $-N\mathscr{D}_N^l(\sqrt{f})$; the second term, by the same change of variable as above, instead satisfies
\begin{align*}
    &\frac{H_s(0)^2}{2\kappa\lambda}\int_{\Omega_N} g(\eta(1))\left[\sqrt{f(\eta^{1-})}+\sqrt{f(\eta)}\right]^2\de\bar\nu_N^\kappa
    \\&\le  \frac{H_s(0)^2}{\kappa\lambda}\int_{\Omega_N} \left\{g(\eta(1))+\bar\varphi_N^\kappa(1)\right\}f(\eta)\de\bar\nu_N^\kappa,
\end{align*}
and so when we sum it with the rightmost term in the {fifth} line of \eqref{eq:sup_all} (with the choice $c=2$)  gives a negative term. The argument for the right boundary term is identical. This ends the proof.
\end{proof}

\section{Proof of the Properties of the Energy Functional}\label{sec:properties}
In this section, we prove Theorem \ref{thm:energy_properties}.
\begin{proof}[Proof of Theorem \ref{thm:energy_properties}] 
We start by noting that the set $C^{0,1}([0, T]\times[0, 1])$ is a dense subset of $L^2_{\kappa, \xi}([0, T]\times[0, 1])$; this follows from a very similar argument to the one given in the proof of \cite[Lemma A.1]{dgn21}. Moreover, by the definition of $\energy_{\kappa, c}(\cdot)$ and by the linearity of $\mathscr{T}_{\kappa, \xi}(\cdot)$, we have that  
\begin{equation*}
    \mathscr{T}_{\kappa, \xi}(H)-c\llangle H, H\rrangle_{\kappa, \xi}\le M_0
\end{equation*} 
and
\begin{equation*}
    -\mathscr{T}_{\kappa, \xi}(H)-c\llangle H, H\rrangle_{\kappa, \xi}\le M_0,
\end{equation*}
which yields
\begin{equation*}
    |\mathscr{T}_{\kappa, \xi}(H)|\le c\llangle H, H\rrangle_{\kappa, \xi}+ M_0.
\end{equation*}
Hence,
\begin{equation*}
    \|\mathscr{T}_{\kappa, \xi}\|:= \mysup_{\substack{H\in C^{0, 1}([0, T]\times[0, 1]) \\ \llangle H, H\rrangle_{\kappa, \xi}=1}}|\mathscr{T}_{\kappa, \xi}(H)|\le c+M_0,
\end{equation*}
so $\mathscr{T}_{\kappa, \xi}$ is a linear function on $C^{0, 1}([0, T]\times[0, 1])$ and it is bounded with respect to the norm of $L^2_{\kappa, \xi}([0, T]\times[0, 1])$. But then, $\mathscr{T}_{\kappa, \xi}(\cdot)$ can be extended to $L^2_{\kappa, \xi}([0, T]\times[0, 1])$, and by Riesz's representation theorem, there exists $f\in L^2_{\kappa, \xi}([0, T]\times[0, 1])$ such that, for all $H\in C^{0, 1}([0, T]\times[0, 1])$, $\mathscr{T}_{\kappa, \xi}(H)=\llangle f, H\rrangle_{\kappa, \xi}$.

\medskip

For every
$H\in C_c^{0,\infty}([0,T]\times(0,1))$,
the boundary terms in the weighted inner product vanish.
Thus, the Riesz representation identity yields
\begin{equation}\label{eq:rie}
\llangle \Phi(\xi)-\bar\varphi^\kappa,\partial_u H\rrangle
=
\llangle fw,H\rrangle,
\end{equation}
where 
\[
w:=\Phi(\xi)+\bar\varphi^\kappa
\]
Integrating the term involving $\bar\varphi^\kappa$ in \eqref{eq:rie} by parts,
we obtain
\[
\llangle \Phi(\xi),\partial_u H\rrangle
=
\llangle fw-\partial_u\bar\varphi^\kappa,H\rrangle
=
-\llangle D,H\rrangle,
\]
where \[
D:=\partial_u\bar\varphi^\kappa-fw
\]

Since   $\xi\in\mathscr B$ and $\Phi$ is Lipschitz, the function
$w$
is bounded. Hence, the Riesz representative
$f\in L^2_{\kappa,\xi}$ satisfies
\[
\int_0^T\!\!\int_0^1 |fw|^2\,du\,dt
\le
\|w\|_{L^\infty((0,T)\times(0,1))}
\int_0^T\!\!\int_0^1 |f|^2w\,du\,dt
<\infty.
\]Moreover, $\partial_u\bar\varphi^\kappa$ is bounded, then the function $D$
belongs to $L^2((0,T)\times(0,1))$.
Therefore, $D$ is the weak spatial derivative of $\Phi(\xi)$,
and we denote
\begin{equation}
    \label{eq: 4,2}
    \partial_u\Phi(\xi)
:=
D=\partial_u\bar\varphi^\kappa-fw.
\end{equation}
In particular, \eqref{eq:rie} can be rewritten as
\begin{equation}\label{eq:ipp_1}
\llangle \Phi(\xi)-\bar\varphi^\kappa,\partial_u H\rrangle
=
-\llangle
\partial_u\Phi(\xi)-\partial_u\bar\varphi^\kappa,H
\rrangle.
\end{equation}
Together with  $\xi\in\mathscr B$ and $\Phi$ is Lipschitz, we have $\Phi(\xi)\in L^2((0,T)\times(0,1))$, and 
this proves that
$\Phi(\xi)\in L^2(0,T;\mathcal H^1)$.

\bigskip

We now establish the energy identity. By the Riesz representation,
for every $H\in C^{0,1}([0,T]\times[0,1])$,
\begin{equation}\label{eq:energy_terms}
\begin{aligned}
\mathscr{T}_{\kappa,\xi}(H)
-c\|H\|_{\kappa,\xi}^2
&=\llangle f,H\rrangle_{\kappa,\xi}
-c\|H\|_{\kappa,\xi}^2\\
&=\frac{1}{4c}\|f\|_{\kappa,\xi}^2
-c\left\|H-\frac{f}{2c}\right\|_{\kappa,\xi}^2,
\end{aligned}
\end{equation}
where $\|H\|_{\kappa,\xi}^2
:=\llangle H,H\rrangle_{\kappa,\xi}$.
Taking the supremum over $H$ yields
\[
\energy_{\kappa,c}(\xi)
\le \frac{1}{4c}\|f\|_{\kappa,\xi}^2.
\]

Conversely, by the density of
$C^{0,1}([0,T]\times[0,1])$ in $L^2_{\kappa,\xi}$,
there exists a sequence $(H^n)_{n\ge1}$ in this test function
space such that
\[
H^n\longrightarrow\frac{f}{2c}
\qquad\text{in }L^2_{\kappa,\xi}.
\]
Using $H^n$ in \eqref{eq:energy_terms}, we obtain
\[
\energy_{\kappa,c}(\xi)
\ge
\frac{1}{4c}\|f\|_{\kappa,\xi}^2
-c\left\|H^n-\frac{f}{2c}\right\|_{\kappa,\xi}^2.
\]
Letting $n\to\infty$, we conclude that
\begin{equation*}
\energy_{\kappa,c}(\xi)
=\frac{1}{4c}\|f\|_{\kappa,\xi}^2.
\end{equation*}
Since, using the equality \eqref{eq: 4,2}, we can write  $f=-(
\partial_u\Phi(\xi)-\partial_u\bar\varphi^\kappa)/w$, then 
\begin{equation*}
\energy_{\kappa,c}(\xi)
=\frac{1}{4c}\Bigg\|\frac{\partial_u\Phi(\xi)-\partial_u\bar\varphi^\kappa}{\Phi(\xi)+\bar\varphi^\kappa}\Bigg\|_{\kappa,\xi}^2.
\end{equation*}

\bigskip

Finally, we identify the boundary quantities obtained from
the Riesz representation. Since
$\Phi(\xi)\in L^2(0,T;\mathcal H^1)$, 
the Sobolev traces are defined, then let us assume that
the values $\Phi(\xi_t(0))$ and $\Phi(\xi_t(1))$ are the traces.
The Sobolev integration by parts formula yields
\begin{equation}\label{eq:T_rewrite_2}
\begin{aligned}
\llangle \Phi(\xi)-\bar\varphi^\kappa,\partial_u H\rrangle
={}&-\llangle
\partial_u\Phi(\xi)-\partial_u\bar\varphi^\kappa,H
\rrangle\\
&+\int_0^T H_t(1)
\bigl(\Phi(\xi_t(1))-\bar\varphi^\kappa(1)\bigr)
\,\de t\\
&-\int_0^T H_t(0)
\bigl(\Phi(\xi_t(0))-\bar\varphi^\kappa(0)\bigr)
\,\de t
\end{aligned}
\end{equation}
for every $H\in C^{0,1}([0,T]\times[0,1])$.

On the other hand, the Riesz representation gives
\begin{equation}\label{eq:T_rewrite}
\begin{aligned}
\llangle \Phi(\xi)-\bar\varphi^\kappa,\partial_u H\rrangle
={}&-\llangle
\partial_u\Phi(\xi)-\partial_u\bar\varphi^\kappa,H
\rrangle\\
&+\int_0^T H_t(0)
\frac{f_t(0)\bigl(\Phi(\xi_t(0))
+\bar\varphi^\kappa(0)\bigr)}
{\kappa\lambda}\,\de t\\
&+\int_0^T H_t(1)
\frac{f_t(1)\bigl(\Phi(\xi_t(1))
+\bar\varphi^\kappa(1)\bigr)}
{\kappa\delta}\,\de t.
\end{aligned}
\end{equation}

Comparing \eqref{eq:T_rewrite_2} and \eqref{eq:T_rewrite},
and choosing $H_t(u)=h(t)(1-u)$ with
$h\in C_c^\infty(0,T)$, we obtain

\begin{equation}
\label{eq.4.7}
\int_0^T h(t)
\left[
\Phi(\xi_t(0))-\bar\varphi^\kappa(0)
+\frac{f_t(0)\bigl(\Phi(\xi_t(0))
+\bar\varphi^\kappa(0)\bigr)}{\kappa\lambda}
\right]\,\de t=0.
\end{equation}

We first show that the expression in brackets belongs to
$L^1(0,T)$. Indeed, since
$\Phi(\xi)\in L^2(0,T;H^1(0,1))$,
the Sobolev trace of $\Phi(\xi)$ at $u=0$ belongs to
$L^2(0,T)$. Consequently,
\[
\Phi(\xi_\cdot(0))-\bar\varphi^\kappa(0)
\in L^2(0,T)\subset L^1(0,T).
\]
Moreover, by the Cauchy--Schwarz inequality,
\[
\begin{aligned}
&\int_0^T |f_t(0)|
\bigl(\Phi(\xi_t(0))+\bar\varphi^\kappa(0)\bigr)\,\de t\\
&\quad\le
\left(
\int_0^T |f_t(0)|^2
\bigl(\Phi(\xi_t(0))+\bar\varphi^\kappa(0)\bigr)\,\de t
\right)^{1/2}\\
&\qquad\times
\left(
\int_0^T
\bigl(\Phi(\xi_t(0))+\bar\varphi^\kappa(0)\bigr)\,\de t
\right)^{1/2}
<\infty,
\end{aligned}
\]
where the finiteness of the first integral follows from
$f\in L^2_{\kappa,\xi}([0,T]\times[0,1])$,
while the second is finite by the Sobolev trace theorem.
Thus, the expression in brackets in \eqref{eq.4.7}
belongs to $L^1(0,T)$.

Since \eqref{eq.4.7} holds for every
$h\in C_c^\infty(0,T)$, the fundamental lemma of the
calculus of variations
(see \cite[Theorem 1.24]{Dacorogna2004})
implies that
\begin{equation}\label{eq:pre_boundary}
f_t(0)\bigl(\Phi(\xi_t(0))+\bar\varphi^\kappa(0)\bigr)
=
-\kappa\lambda
\bigl(\Phi(\xi_t(0))-\bar\varphi^\kappa(0)\bigr)
\end{equation}
for almost every $t\in(0,T)$.

Similarly, choosing $H_t(u)=h(t)u$, we obtain
\[
f_t(1)\bigl(\Phi(\xi_t(1))+\bar\varphi^\kappa(1)\bigr)
=
\kappa\delta
\bigl(\Phi(\xi_t(1))-\bar\varphi^\kappa(1)\bigr)
\]
for almost every $t\in(0,T)$.

Since
\[
\begin{aligned}
\partial_u\bar\varphi^\kappa(0)
&=\kappa\bigl(\lambda\bar\varphi^\kappa(0)-\alpha\bigr)\quad\mbox{and }\\
\partial_u\bar\varphi^\kappa(1)
&=\kappa\bigl(\beta-\delta\bar\varphi^\kappa(1)\bigr),
\end{aligned}
\]
we conclude that
\[
\begin{aligned}
&q_0(t):=\partial_u\bar\varphi^\kappa(0)
-f_t(0)\bigl(\Phi(\xi_t(0))
+\bar\varphi^\kappa(0)\bigr)=\kappa\bigl(\lambda \Phi(\xi_t(0))-\alpha\bigr)\quad\mbox{and }\\
&q_1(t):=\partial_u\bar\varphi^\kappa(1)
-f_t(1)\bigl(\Phi(\xi_t(1))
+\bar\varphi^\kappa(1)\bigr)=\kappa\bigl(\beta-\delta \Phi(\xi_t(1))\bigr)
\end{aligned}
\]
for almost every $t\in(0,T)$.

Note that the identities above are precisely the Robin
boundary conditions prescribed in \eqref{eq:zre_robin},
interpreted in a generalized sense.
Indeed, since we only know that
$\partial_u\Phi(\xi)\in L^2((0,T)\times(0,1))$,
the boundary traces of $\partial_u\Phi(\xi)$
are not necessarily well-defined in the Sobolev sense.
Nevertheless, the Riesz representation provides
well-defined boundary quantities $q_0$ and $q_1$,
which satisfy the prescribed Robin boundary conditions.
We may therefore interpret $q_0$ and $q_1$ as
generalized boundary values of $\partial_u\Phi(\xi)$
at $u=0$ and $u=1$, respectively.
This interpretation does not imply that they coincide
with the Sobolev traces of $\partial_u\Phi(\xi)$,
which would require additional regularity and
a compatibility argument.

\end{proof}

\section{Proof of the Convergence of Weak Solutions}\label{sec:convergence}

In this section, we prove Theorem \ref{thm:convergence}. We first present a proof highlighting the main steps, and we then proceed to state and prove all the required auxiliary results.

\subsection{Proof of Theorem \ref{thm:convergence}}

\begin{proof}[Proof of Theorem \ref{thm:convergence}] Our proof is divided into two main steps. We first show that, for any subsequence of $\{\Phi(\rho^{\kappa})\}_{\kappa>0}$, we can extract a subsequence $\{\Phi(\rho^{\kappa_j})\}_{j\in\mathbb{N}}$ which converges strongly in $L^2([0, T]\times[0, 1])$, and that this implies that $\{\rho^{\kappa_j}\}_{j\in\mathbb{N}}$ also converges strongly in $L^2([0, T]\times[0, 1])$. Next, we prove that any limit point $\rho^*$ of $\{\rho^\kappa\}_{\kappa>0}$ is such that $\Phi(\rho^*)\in {{L^2(0,T; \mathcal{H}^1)}}$, satisfies the appropriate integral formulation and, in the case $\kappa\to\infty$, the appropriate boundary conditions.

\medskip
\noindent\textbf{Extraction of Strongly Convergent Subsequences.}
By Lemma \ref{lemma:Phi_L2_H1}, we have
\[
    \sup_{\kappa>0}
    \|\Phi(\rho^\kappa)\|_{L^2(0,T;\mathcal H^1)}
    <\infty.
\]
Moreover, Lemma \ref{lemma:time_reg} gives
\[
    \|\partial_t\rho^\kappa\|_{L^2(0,T;\mathcal H^{-1})}
    \le
    \|\partial_u\Phi(\rho^\kappa)\|_{L^2([0,T]\times[0,1])},
\]
where $\mathcal H^{-1}=(\mathcal H_0^1)^*$.
Consequently, Lemma \ref{lemma:Phi_L2_H1} also yields
\[
    \sup_{\kappa>0}
    \|\partial_t\rho^\kappa\|_{L^2(0,T;\mathcal H^{-1})}
    <\infty.
\]
Thus, the hypotheses of Lemma \ref{lemma:time_translations} are satisfied, and the family
$\{\Phi(\rho^\kappa)\}_{\kappa>0}$ is relatively compact in
$L^2([0,T]\times[0,1])$.

Therefore, from any sequence of positive parameters, we can
extract a subsequence $(\kappa_j)_{j\in\mathbb N}$ such that
$\Phi(\rho^{\kappa_j})$ converges strongly in
$L^2([0,T]\times[0,1])$ to some function $F$.
By Proposition \ref{prop:comparison}, there exists a constant $c_0<\infty$,
independent of $\kappa$, such that
$0\le\rho^\kappa\le c_0$ almost everywhere.
Hence,
\[
    \Phi(0)\le F\le\Phi(c_0)
    \qquad\text{almost everywhere}.
\]
We may therefore define $\rho^*:=\Phi^{-1}(F)$, so that
\[
    \Phi(\rho^{\kappa_j})\longrightarrow\Phi(\rho^*)
    \qquad\text{strongly in }L^2([0,T]\times[0,1]).
\]

 In particular, denoting by $\mu$ the Lebesgue measure on $[0, T]\times[0, 1]$, the sequence $\{\Phi(\rho^{\kappa_j})\}_{j\in\mathbb{N}}$ converges to $\Phi(\rho^*)$ in $\mu$-measure. Since, for any finite measure, convergence in measure is preserved by continuous functions and since $u\mapsto \Phi^{-1}(u)^2$ is continuous, this implies that the sequence $\{(\rho^{\kappa_j})^2\}_{j\in\mathbb{N}}$ converges to $(\rho^*)^2$ in $\mu$-measure. Moreover, by Lemma \ref{lemma:rho_Lp}, the sequence $\{(\rho^{\kappa_j})^2\}_{j\in\mathbb{N}}$ is bounded in $L^p([0, T]\times[0, 1])$ for any $p>0$, and hence it is uniformly integrable: thus, by the Lebesgue-Vitali Theorem \cite[Theorem 4.5.4]{bog07}, $\{(\rho^{\kappa_j})^2\}_{j\in\mathbb{N}}$ converges to $(\rho^*)^2$ in $L^1(\mu)$, or equivalently, the sequence $\{\rho^{\kappa_j}\}_{j\in\mathbb{N}}$ converges to $\rho^*$ strongly in $L^2([0, T]\times[0, 1])$.

\medskip
\textbf{Any Limit Point Solves the Nonlinear PDE.} Now, by Lemma \ref{lemma:Phi*_L2_H1}, we have that $\Phi(\rho^*)\in {{L^2(0,T; \mathcal{H}^1)}}.$ Moreover, in the case $\kappa_j\to\infty$ as $j\to\infty$, by Lemma \ref{lemma:rho*_boundary} we have that $\rho^*$ satisfies the boundary conditions given in item iv) of Definition \ref{def:weak_dirichlet}. Hence, it only remains to show that $\rho^*$ satisfies the integral formulation of the appropriate PDE, namely item iii) of Definition \ref{def:weak_neumann} when $\kappa_j\to0$, and of Definition \ref{def:weak_dirichlet} when $\kappa_j\to\infty$. By performing an integration by parts, we can rewrite the integral formulation \eqref{eq:integral_zre_rob} as 
\begin{equation} \label{eq:integral_ibp_1}
    \begin{split}
    \langle\rho^{\kappa_j}_t, G_t\rangle &- \langle\gamma, G_0\rangle - \int_0^t\langle \rho^{\kappa_j}_s, \partial_s G_s\rangle \de s +\int_0^t \langle\partial_u \Phi(\rho^{\kappa_j}_s), \partial_u G_s\rangle\de s
    \\&-\kappa_j\int_0^t \big\{(\beta-\delta\Phi(\rho^{\kappa_j}_s(1)))G_s(1) - (\lambda \Phi(\rho^{\kappa_j}_s(0))-\alpha)G_s(0) \big\}\de s = 0.
    \end{split}
\end{equation}
for all $t\in[0, T]$ and $G\in C^{1, 2}([0, T]\times[0, 1])$. Now, in the case $\kappa_j\to\infty$, recall that in Definition \ref{def:weak_dirichlet} we only consider functions $G\in C^{1, 2}_0([0, T]\times[0, 1])$, so that the second line of \eqref{eq:integral_ibp_1} vanishes. In the case $\kappa_j\to0$, instead, by the standard inequality $\|f\|_\infty\le \|f\|_2+\|\partial_u f\|_2$ applied to $f=\Phi(\rho^{\kappa_j}_s)$ and the Cauchy-Schwarz inequality,
\begin{align*}
    \kappa_j &\left| \int_0^t \big\{(\beta-\delta\Phi(\rho^{\kappa_j}_r(1)))G_s(1) - (\lambda \Phi(\rho^{\kappa_j}_r(0))-\alpha)G_s(0) \big\}\de s\right|
    \\&\le \kappa_j (\alpha+\beta) \|G\|_\infty T
    \\&\phantom{\le}+\kappa_j (\lambda+\delta)\|G\|_\infty \left|\int_0^t \left\{\|\Phi(\rho^{\kappa_j}_s)\|_2+\|\partial_u\Phi(\rho^{\kappa_j}_s)\|_2\right\}\de s\right|
    \\&\le \kappa_j \|G\|_\infty \left\{(\alpha+\beta)T+(\lambda+\delta)\sqrt{T}\|\Phi(\rho^{\kappa_j})\|_{{{L^2(0,T; \mathcal{H}^1)}}}\right\}.
\end{align*}
But now, by Lemma \ref{lemma:Phi_L2_H1}, $\{\|\Phi(\rho^{\kappa_j})\|_{{{L^2(0,T; \mathcal{H}^1)}}}\}_{j\in\mathbb{N}}$ is bounded, and hence the last line above vanishes as $j\to\infty$.

It only remains to analyse the first line of \eqref{eq:integral_ibp_1}. This is achieved by Lemmas \ref{lemma:rho*_weak} and \ref{lemma:Phi*_weak}, which guarantee that, as $j\to\infty$, it converges to
\begin{equation*}
    \langle\rho^*_t, G_t\rangle - \langle\gamma, G_0\rangle - \int_0^t\langle \rho^*_s, \partial_s G_s\rangle \de s +\int_0^t \langle\partial_u \Phi(\rho^*_s), \partial_u G_s\rangle\de s,
\end{equation*}
which in turn, by performing an integration by parts, we can rewrite as
\begin{equation} \label{eq:proof_neumann}
    \begin{split}
    \langle\rho^*_t, G_t\rangle &- \langle\gamma, G_0\rangle - \int_0^t\{\langle \rho^*_s, \partial_s G_s\rangle + \langle\Phi(\rho^*_s), \Delta G_s\rangle\}\de s
    \\&+\int_0^t \{\Phi(\rho^*_s(1))\partial_u G_s(1)-\Phi(\rho^*_s(0))\partial_u G_s(0) \}\de s.
    \end{split}
\end{equation}
Since weak solutions to \eqref{eq:zre_neumann} are unique (see Proposition \ref{prop:uniqueness_zre}), this completes the proof for the Neumann case. As for the Dirichlet case, it suffices to apply Lemma \ref{lemma:rho*_boundary}, which guarantees that $\Phi(\rho^*_t(0))=\frac{\alpha}{\lambda}$ and $\Phi(\rho^*_t(1))=\frac{\beta}{\delta}$ for almost every $t\in(0, T]$, and allows us to rewrite \eqref{eq:proof_neumann} as
\begin{align*} 
    \langle\rho^*_t, G_t\rangle &- \langle\gamma, G_0\rangle - \int_0^t\{\langle \rho^*_s, \partial_s G_s\rangle + \langle\Phi(\rho^*_s), \Delta G_s\rangle\}\de s
    \\&+\int_0^t \left\{\frac{\beta}{\delta}\partial_u G_s(1)-\frac{\alpha}{\lambda}\partial_u G_s(0) \right\}\de s.
\end{align*}
Since weak solutions to \eqref{eq:zre_dirichlet} are unique (Proposition \ref{prop:uniqueness_zre}), the proof is completed.
\end{proof}

\subsection{Auxiliary Results}
In this section, we state and prove some auxiliary results which we used in the proof of Theorem \ref{thm:convergence} above. We start by showing that, under the assumptions of the theorem, the family $\{\Phi(\rho^\kappa)\}_{\kappa>0}$ is bounded in ${{L^2(0,T; \mathcal{H}^1)}}$.

\begin{lemma}\label{lemma:Phi_L2_H1}
Let $\gamma:[0,1]\to[0,\infty)$ be a measurable function
satisfying \eqref{eq:initial_condition}.
For each $\kappa>0$, let $\rho^\kappa$ be the unique weak
solution to the Robin problem \eqref{eq:zre_robin}  with initial profile $\gamma$.
Then the family $\{\Phi(\rho^\kappa)\}_{\kappa>0}$ is bounded
in $L^2(0,T;\mathcal H^1)$; that is,
\[
    \sup_{\kappa>0}
    \|\Phi(\rho^\kappa)\|_{L^2(0,T;\mathcal H^1)}<\infty.
\]
\end{lemma}
\begin{proof}
Note that, by \eqref{eq:Phi_g*},
\begin{align*}
    \left\|\Phi(\rho^\kappa)\right\|_2^2&=\int_0^T\int_0^1 \Phi(\rho^\kappa_t(u))^2\de u\de t
    \\&\le 2\int_0^T\int_0^1 \left\{\Phi(\rho^\kappa_t(u))-\Phi(0)\right\}^2\de u\de t + 2\int_0^T\int_0^1 \Phi(0)^2\de u\de t
    \\&\le 2(g^*)^2\int_0^T\int_0^1 \rho^\kappa_t(u)^2\de u\de t + 2T\Phi(0)^2.
\end{align*}
By Proposition \ref{prop:comparison}, this implies that $\{\|\Phi(\rho^\kappa)\|_2^2\}_\kappa$ is bounded. Now, note that
\begin{align*}
    \|\partial_u&\Phi(\rho^\kappa)\|_2^2
    \\&= \llangle \partial_u\Phi(\rho^\kappa), \partial_u\Phi(\rho^\kappa)\rrangle
    \\&\le 2\llangle \partial_u\Phi(\rho^\kappa)-\partial_u\bar\varphi^\kappa,\partial_u\Phi(\rho^\kappa)-\partial_u\bar\varphi^\kappa\rrangle+2\llangle \partial_u\bar\varphi^\kappa, \partial_u\bar\varphi^\kappa\rrangle.
\end{align*}
Multiplying and dividing by $\Phi(\rho^\kappa(u))+\bar\varphi^\kappa(u)$ inside the first bracket in the right hand side above, we have
\begin{align*}
    &\llangle \partial_u\Phi(\rho^\kappa)-\partial_u\bar\varphi^\kappa,\partial_u\Phi(\rho^\kappa)-\partial_u\bar\varphi^\kappa\rrangle
    \\&\le \mysup_{u\in[0, 1]}\{\Phi(\rho^\kappa(u))+\bar\varphi^\kappa(u)\}\Big\llangle \frac{\partial_u\Phi(\rho^\kappa)-\partial_u\bar\varphi^\kappa}{\Phi(\rho^\kappa)+\bar\varphi^\kappa}, \partial_u\Phi(\rho^\kappa)-\partial_u\bar\varphi^\kappa\Big\rrangle.
\end{align*}
Now, remembering the definition of the norm $\|\cdot\|_{\kappa, \rho^\kappa}$ and \eqref{eq:energy_id}, 
we can bound from above the last bracket above as
\begin{align*}
&\Big\llangle \frac{\partial_u\Phi(\rho^\kappa)-\partial_u\bar\varphi^\kappa}{\Phi(\rho^\kappa)+\bar\varphi^\kappa}, \partial_u\Phi(\rho^\kappa)-\partial_u\bar\varphi^\kappa\Big\rrangle\\=
    &\Big\llangle \frac{\partial_u\Phi(\rho^\kappa)-\partial_u\bar\varphi^\kappa}{\Phi(\rho^\kappa)+\bar\varphi^\kappa}, \frac{\partial_u\Phi(\rho^\kappa)-\partial_u\bar\varphi^\kappa} {\Phi(\rho^\kappa)+\bar\varphi^\kappa} \{\Phi(\rho^\kappa)+\bar\varphi^\kappa\}\Big\rrangle
\\\leq &
   \Big\llangle \frac{\partial_u\Phi(\rho^\kappa)-\partial_u\bar\varphi^\kappa}{\Phi(\rho^\kappa)+\bar\varphi^\kappa}, \frac{\partial_u\Phi(\rho^\kappa)-\partial_u\bar\varphi^\kappa} {\Phi(\rho^\kappa)+\bar\varphi^\kappa}\Big\rrangle_{\kappa, \rho^\kappa}=\;4c\,\energy_{\kappa,c}(\rho^\kappa).
\end{align*}
 By Theorem \ref{thm:energy_estimate}, $\energy_{\kappa,c}(\rho^\kappa)\leq M_0$.
Thus, 
\begin{align*}
    \|\partial_u&\Phi(\rho^\kappa)\|_2^2
    \le\; 8c\,M_0\,\mysup_{u\in[0, 1]}\{\Phi(\rho^\kappa(u))+\bar\varphi^\kappa(u)\}+2\llangle \partial_u\bar\varphi^\kappa, \partial_u\bar\varphi^\kappa\rrangle.
\end{align*}
But now, recalling the definition of $\bar\varphi^\kappa$ in \eqref{eq:limit_fug}, by Proposition \ref{prop:comparison} and since $\Phi$ is strictly increasing,
\begin{equation}\label{eq:necessary_Linfty}
    \mysup_{\kappa>0}\mysup_{u\in[0, 1]}\{\Phi(\rho^\kappa(u))+\bar\varphi^\kappa(u)\}\le \left(\Phi(c_0)+\mymax\left\{\frac{\alpha}{\lambda}, \frac{\beta}{\delta}\right\}\right)
\end{equation}
for some constant $c_0<\infty$. Moreover,
\begin{align*}
    \mysup_{\kappa>0}\llangle \partial_u\bar\varphi^\kappa, \partial_u\bar\varphi^\kappa\rrangle&= T\mysup_{\kappa>0} \left(\frac{\kappa\beta\lambda-\kappa\alpha\delta}{\kappa\lambda\delta+\lambda+\delta}\right)^2\le T \left(\frac{\beta\lambda-\alpha\delta}{\lambda\delta}\right)^2.
\end{align*}
Hence, we get
\begin{equation}\label{eq:partialPhi_bdd}
    \|\partial_u\Phi(\rho^\kappa)\|_2^2 \le  8cM_0 \left(\Phi(c_0)+\mymax\left\{\frac{\alpha}{\lambda}, \frac{\beta}{\delta}\right\}\right) + 2T\left(\frac{\beta\lambda-\alpha\delta}{\lambda\delta}\right)^2,
\end{equation}
where $c, c_0, M_0<\infty$ do not depend on $\kappa$. This completes the proof.
\end{proof}

\begin{remark}\label{remark:necessary_Linfty}
 Lemma \ref{lemma:Phi_L2_H1} requires the family $\{\Phi(\rho^\kappa)\}_\kappa$ to be bounded in $L^\infty([0, T]\times[0, 1])$; see equation \eqref{eq:necessary_Linfty} {{and it comes from the fact  that $\{\rho^\kappa\}_\kappa$ is bounded in $L^\infty([0,T]\times[0,1])$, as established by the comparison principle (Proposition \ref{prop:comparison}).
}}

\end{remark}

We now define the Sobolev space $\mathcal{H}_0^1$ via
\begin{equation*}
    \mathcal{H}_0^1:=\left\{\varphi\in\mathcal{H}^1: \varphi(0)=\varphi(1)=0\right\},
\end{equation*}
and the Sobolev space 
$\mathcal{H}^{-1}$ as the dual of $\mathcal{H}_0^1$, equipped with the norm
\begin{equation}\label{eq:H-1_norm}
    \|\cdot\|_{\mathcal{H}^{-1}}:=\mysup_{\varphi\in \mathcal{H}_0^1, \|\varphi\|_{\mathcal{H}^1}=1} \left| \langle \cdot, \varphi  \rangle \right|.
\end{equation}

\begin{lemma}\label{lemma:time_reg}
For every $\kappa>0$, the distributional time derivative
of $\rho^\kappa$ belongs to
$L^2(0,T;\mathcal H^{-1})$, where
$\mathcal H^{-1}=(\mathcal H_0^1)^*$.
More precisely, for almost every $t\in(0,T)$,
\[
\partial_t\rho^\kappa_t[\varphi]
=
-\langle\partial_u\Phi(\rho^\kappa_t),
\partial_u\varphi\rangle,
\qquad \varphi\in\mathcal H_0^1.
\]
In particular,
\begin{equation}\label{eq:rho_L2_H-1}
\|\partial_t\rho^\kappa\|_
{L^2(0,T;\mathcal H^{-1})}^2
\leq
\int_0^T
\|\partial_u\Phi(\rho^\kappa_t)\|_{L^2(0,1)}^2
\,dt
<\infty.
\end{equation}
\end{lemma}

\begin{proof}
Fix $\kappa>0$. Since $\rho^\kappa$ is a weak
solution of \eqref{eq:integral_zre_rob}, testing
the weak formulation against
$G\in C_c^\infty((0,T)\times(0,1))$
and integrating by parts in space yields
\begin{equation}\label{eq:time_reg_weak}
-\int_0^T
\langle\rho^\kappa_t,\partial_tG_t\rangle\,dt
=
-\int_0^T
\langle\partial_u\Phi(\rho^\kappa_t),
\partial_uG_t\rangle\,dt.
\end{equation}

For almost every $t\in(0,T)$, define the
linear functional
$F_t:\mathcal H_0^1\to\mathbb R$ by
\[
F_t[\varphi]
:=
-\langle\partial_u\Phi(\rho^\kappa_t),
\partial_u\varphi\rangle.
\]
By the Cauchy--Schwarz inequality,
\[
\begin{aligned}
|F_t[\varphi]|
&\leq
\|\partial_u\Phi(\rho^\kappa_t)\|_{L^2(0,1)}
\|\partial_u\varphi\|_{L^2(0,1)}
\\
&\leq
\|\partial_u\Phi(\rho^\kappa_t)\|_{L^2(0,1)}
\|\varphi\|_{\mathcal H_0^1}.
\end{aligned}
\]
Consequently, $F_t\in\mathcal H^{-1}$ and
\[
\|F_t\|_{\mathcal H^{-1}}
\leq
\|\partial_u\Phi(\rho^\kappa_t)\|_{L^2(0,1)}.
\]

Since
$\Phi(\rho^\kappa)\in
L^2(0,T;\mathcal H^1)$,
we obtain
\[
\int_0^T
\|F_t\|_{\mathcal H^{-1}}^2\,dt
\leq
\int_0^T
\|\partial_u\Phi(\rho^\kappa_t)\|_{L^2(0,1)}^2
\,dt
<\infty.
\]

Finally, identity \eqref{eq:time_reg_weak}
shows that $F$ coincides with the
distributional time derivative of
$\rho^\kappa$. Therefore,
\[
\partial_t\rho^\kappa
\in L^2(0,T;\mathcal H^{-1}),
\]
and \eqref{eq:rho_L2_H-1} follows.
\end{proof}

\begin{lemma}\label{lemma:time_translations}
Let $\{\rho^\kappa\}_{\kappa\in\mathbb N}$
be a sequence of nonnegative functions such that
\begin{equation}\label{eq:time_translations_assumptions}
\sup_{\kappa\in\mathbb N}
\left\{
\|\Phi(\rho^\kappa)\|_{L^2(0,T;\mathcal H^1)}
+
\|\partial_t\rho^\kappa\|_
{L^2(0,T;\mathcal H^{-1})}
\right\}<\infty,
\end{equation}
where $\mathcal H^{-1}=(\mathcal H_0^1)^*$.
Assume that $\Phi$ is nondecreasing and
Lipschitz continuous.

Then there exists a constant $C>0$,
independent of $\kappa$ and $h$, such that
\begin{equation}\label{eq:time_translations_bound}
\int_0^{T-h}
\|\Phi(\rho^\kappa_{t+h})
-\Phi(\rho^\kappa_t)\|_{L^2(0,1)}^2\,dt
\leq Ch^{1/2}
\end{equation}
for every $h\in(0,T)$.

Consequently,
$\{\Phi(\rho^\kappa)\}_{\kappa\in\mathbb N}$
is relatively compact in
$L^2((0,T)\times(0,1))$.
\end{lemma}

\begin{proof}
Let $L>0$ be a Lipschitz constant for $\Phi$.
By assumption \eqref{eq:time_translations_assumptions},
there exist constants $M_1,M_2>0$ such that
\begin{equation}\label{eq:uniform_constants}
\sup_\kappa
\int_0^T
\|\Phi(\rho^\kappa_t)\|_{\mathcal H^1}^2\,dt
\leq M_1
\end{equation}
and
\begin{equation}\label{eq:uniform_time_constant}
\sup_\kappa
\int_0^T
\|\partial_t\rho^\kappa_t\|_{\mathcal H^{-1}}^2\,dt
\leq M_2.
\end{equation}

Fix $h\in(0,T)$ and
$\varepsilon\in(0,1/4)$.
Choose $\chi_\varepsilon\in C_c^\infty(0,1)$
such that
\[
0\leq\chi_\varepsilon\leq1,
\qquad
\chi_\varepsilon=1
\quad\text{on }[2\varepsilon,1-2\varepsilon],
\]
and
\[
\|\partial_u\chi_\varepsilon\|_{L^\infty(0,1)}
\leq \frac{K_\chi}{\varepsilon},
\]
where $K_\chi>0$ is independent of
$\varepsilon$. 
Since $\Phi$ is nondecreasing and Lipschitz
continuous, for all $a,b\geq0$,
\begin{equation}\label{eq:monotone_lipschitz}
|\Phi(a)-\Phi(b)|^2
\leq L(a-b)(\Phi(a)-\Phi(b)).
\end{equation}
Applying \eqref{eq:monotone_lipschitz}
pointwise, multiplying by
$\chi_\varepsilon^2$, and integrating,
we obtain
\begin{equation}\label{eq:cutoff_duality}
\begin{aligned}
&\int_0^{T-h}\int_0^1
\chi_\varepsilon^2
|\Phi(\rho^\kappa_{t+h})
-\Phi(\rho^\kappa_t)|^2\,du\,dt
\\
&\quad\leq
L\int_0^{T-h}
\Big\langle
\rho^\kappa_{t+h}-\rho^\kappa_t,
\chi_\varepsilon^2
\bigl(
\Phi(\rho^\kappa_{t+h})
-\Phi(\rho^\kappa_t)
\bigr)
\Big\rangle\,dt.
\end{aligned}
\end{equation}
The bracket above is well defined between
$\mathcal H^{-1}$ and $\mathcal H_0^1$,
because the cutoff function has compact
support in $(0,1)$.

By the duality inequality and
the Cauchy--Schwarz inequality in time,
the right-hand side of
\eqref{eq:cutoff_duality} is bounded by
\begin{equation}\label{eq:cutoff_CS}
\begin{aligned}
&L
\|\rho^\kappa(\cdot+h)-\rho^\kappa\|_
{L^2(0,T-h;\mathcal H^{-1})}
\;\times\,
\left\|
\chi_\varepsilon^2
\bigl(
\Phi(\rho^\kappa(\cdot+h))
-\Phi(\rho^\kappa)
\bigr)
\right\|_
{L^2(0,T-h;\mathcal H_0^1)}.
\end{aligned}
\end{equation}
We estimate these two factors separately.
First, since
$\partial_t\rho^\kappa
\in L^2(0,T;\mathcal H^{-1})$,
we have
\[
\rho^\kappa_{t+h}-\rho^\kappa_t
=
\int_t^{t+h}
\partial_s\rho^\kappa_s\,ds
\]
in $\mathcal H^{-1}$.
By the Cauchy-Schwarz inequality,
\[
\begin{aligned}
&\|\rho^\kappa_{t+h}-\rho^\kappa_t\|_
{\mathcal H^{-1}}^2
\leq
h\int_t^{t+h}
\|\partial_s\rho^\kappa_s\|_
{\mathcal H^{-1}}^2\,ds.
\end{aligned}
\]
Integrating in time and applying Fubini's
theorem yields
\[
\begin{aligned}
&\int_0^{T-h}
\|\rho^\kappa_{t+h}-\rho^\kappa_t\|_
{\mathcal H^{-1}}^2\,dt
\leq
h^2\int_0^T
\|\partial_s\rho^\kappa_s\|_
{\mathcal H^{-1}}^2\,ds
\leq M_2h^2,
\end{aligned}
\]
where the last inequality follows from
\eqref{eq:uniform_time_constant}.
Consequently,
\begin{equation}\label{eq:rho_time_shift}
\|\rho^\kappa(\cdot+h)-\rho^\kappa\|_
{L^2(0,T-h;\mathcal H^{-1})}
\leq \sqrt{M_2}\,h.
\end{equation}

We now estimate the second factor in
\eqref{eq:cutoff_CS}.
By the product rule,
\[
\begin{aligned}
&\quad\partial_u\left[
\chi_\varepsilon^2
\bigl(
\Phi(\rho^\kappa_{t+h})
-\Phi(\rho^\kappa_t)
\bigr)
\right]\\&
=
2\chi_\varepsilon
(\partial_u\chi_\varepsilon)
\bigl(
\Phi(\rho^\kappa_{t+h})
-\Phi(\rho^\kappa_t)
\bigr)
+
\chi_\varepsilon^2
\partial_u\bigl(
\Phi(\rho^\kappa_{t+h})
-\Phi(\rho^\kappa_t)
\bigr).
\end{aligned}
\]

Using the bounds on $\chi_\varepsilon$,
we obtain
\[
\begin{aligned}
&\left\|
\chi_\varepsilon^2
\bigl(
\Phi(\rho^\kappa_{t+h})
-\Phi(\rho^\kappa_t)
\bigr)
\right\|_{\mathcal H_0^1}
\\
&\quad\leq
\left(1+\frac{2K_\chi}{\varepsilon}\right)
\left(
\|\Phi(\rho^\kappa_{t+h})\|_{\mathcal H^1}
+
\|\Phi(\rho^\kappa_t)\|_{\mathcal H^1}
\right).
\end{aligned}
\]
Squaring, integrating in time, and using
$(a+b)^2\leq2(a^2+b^2)$, we get
\[
\begin{aligned}
&\int_0^{T-h}
\left\|
\chi_\varepsilon^2
\bigl(
\Phi(\rho^\kappa_{t+h})
-\Phi(\rho^\kappa_t)
\bigr)
\right\|_{\mathcal H_0^1}^2\,dt
\\
&\quad\leq
2\left(1+\frac{2K_\chi}{\varepsilon}\right)^2
\Bigg[
\int_0^{T-h}
\|\Phi(\rho^\kappa_{t+h})\|_{\mathcal H^1}^2\,dt
\\
&\hspace{90pt}+
\int_0^{T-h}
\|\Phi(\rho^\kappa_t)\|_{\mathcal H^1}^2\,dt
\Bigg].
\end{aligned}
\]
By the change of variables $s=t+h$,
\[
\begin{aligned}
&\int_0^{T-h}
\|\Phi(\rho^\kappa_{t+h})\|_{\mathcal H^1}^2\,dt
\\
&\quad=
\int_h^T
\|\Phi(\rho^\kappa_s)\|_{\mathcal H^1}^2\,ds
\leq M_1.
\end{aligned}
\]
The second integral is also bounded by
$M_1$. Hence,
\[
\begin{aligned}
&\int_0^{T-h}
\left\|
\chi_\varepsilon^2
\bigl(
\Phi(\rho^\kappa_{t+h})
-\Phi(\rho^\kappa_t)
\bigr)
\right\|_{\mathcal H_0^1}^2\,dt
\leq
4M_1
\left(1+\frac{2K_\chi}{\varepsilon}\right)^2.
\end{aligned}
\]
Since $\varepsilon<1/4$,
\[
1+\frac{2K_\chi}{\varepsilon}
\leq
\frac{1+2K_\chi}{\varepsilon}.
\]
Taking square roots,
\begin{equation}\label{eq:cutoff_H1_bound}
\begin{aligned}
&\left\|
\chi_\varepsilon^2
\bigl(
\Phi(\rho^\kappa(\cdot+h))
-\Phi(\rho^\kappa)
\bigr)
\right\|_
{L^2(0,T-h;\mathcal H_0^1)}
\\
&\quad\leq
\frac{2(1+2K_\chi)\sqrt{M_1}}
{\varepsilon}.
\end{aligned}
\end{equation}

Combining \eqref{eq:cutoff_CS},
\eqref{eq:rho_time_shift}, and
\eqref{eq:cutoff_H1_bound}, we conclude that
\begin{equation}\label{eq:interior_time_shift}
\begin{aligned}
&\int_0^{T-h}\int_0^1
\chi_\varepsilon^2
|\Phi(\rho^\kappa_{t+h})
-\Phi(\rho^\kappa_t)|^2\,du\,dt
\\
&\quad\leq
2L(1+2K_\chi)\sqrt{M_1M_2}
\frac{h}{\varepsilon}.
\end{aligned}
\end{equation}

It remains to control the contribution
near the boundary.
Since $\chi_\varepsilon=1$ on
$[2\varepsilon,1-2\varepsilon]$,
the support of $1-\chi_\varepsilon^2$
has Lebesgue measure at most
$4\varepsilon$. Then for each t
\[
\begin{aligned}
&\int_0^1(1-\chi_\varepsilon^2)
|\Phi(\rho^\kappa_{t+h})-\Phi(\rho^\kappa_t)|^2\,du\\
&\quad\leq
4\varepsilon
\|\Phi(\rho^\kappa_{t+h})-\Phi(\rho^\kappa_t)\|_{L^\infty(0,1)}^2.
\end{aligned}
\]
By the one-dimensional Sobolev embedding, as established
previously, we have
\[
\|f\|_{L^\infty(0,1)}^2
\leq 2\|f\|_{\mathcal H^1}^2
\]
for any function $f\in \mathcal H^1. $
Applying this inequality to the difference, we obtain
\[
\begin{aligned}
&4\varepsilon
\|\Phi(\rho^\kappa_{t+h})
-\Phi(\rho^\kappa_t)\|_{L^\infty(0,1)}^2
\leq
8\varepsilon
\|\Phi(\rho^\kappa_{t+h})
-\Phi(\rho^\kappa_t)\|_{\mathcal H^1}^2.
\end{aligned}
\]
Next, using the inequality
\(\|f-g\|_{\mathcal H^1}^2
\leq
2\left(
\|f\|_{\mathcal H^1}^2
+
\|g\|_{\mathcal H^1}^2
\right),\)
we obtain
\[
\begin{aligned}
&\int_0^1
(1-\chi_\varepsilon^2)
|\Phi(\rho^\kappa_{t+h})
-\Phi(\rho^\kappa_t)|^2\,du
\\
&\quad\leq
16\varepsilon
\left(
\|\Phi(\rho^\kappa_{t+h})\|_{\mathcal H^1}^2
+
\|\Phi(\rho^\kappa_t)\|_{\mathcal H^1}^2
\right).
\end{aligned}
\]
Therefore,
\[
\begin{aligned}
&\int_0^{T-h}\int_0^1
(1-\chi_\varepsilon^2)
|\Phi(\rho^\kappa_{t+h})
-\Phi(\rho^\kappa_t)|^2\,du\,dt
\\
&\quad\leq
16\varepsilon
\int_0^{T-h}
\left(
\|\Phi(\rho^\kappa_{t+h})\|_{\mathcal H^1}^2
+
\|\Phi(\rho^\kappa_t)\|_{\mathcal H^1}^2
\right)\,dt
\;\leq\;
32M_1\varepsilon.
\end{aligned}
\]

Adding this estimate to
\eqref{eq:interior_time_shift}, we obtain
\[
\begin{aligned}
&\int_0^{T-h}
\|\Phi(\rho^\kappa_{t+h})
-\Phi(\rho^\kappa_t)\|_{L^2(0,1)}^2\,dt
\\
&\quad\leq
A\frac{h}{\varepsilon}+B\varepsilon,
\end{aligned}
\]
where
\[
A:=2L(1+2K_\chi)\sqrt{M_1M_2},
\qquad
B:=32M_1.
\]

For $0<h<1/64$, choose
$\varepsilon=\sqrt h$.
Then
\[
\begin{aligned}
&\int_0^{T-h}
\|\Phi(\rho^\kappa_{t+h})
-\Phi(\rho^\kappa_t)\|_{L^2(0,1)}^2\,dt
\\
&\quad\leq (A+B)\sqrt h.
\end{aligned}
\]
For $h\geq1/64$, the uniform
$L^2((0,T)\times(0,1))$ bound gives
\[
\begin{aligned}
&\int_0^{T-h}
\|\Phi(\rho^\kappa_{t+h})
-\Phi(\rho^\kappa_t)\|_{L^2(0,1)}^2\,dt
\\
&\quad\leq4M_1
\leq32M_1\sqrt h.
\end{aligned}
\]

Thus, setting
\[
C_{\mathrm{time}}
:=
\max\{A+B,32M_1\},
\]
we conclude that
\begin{equation*}
\int_0^{T-h}
\|\Phi(\rho^\kappa_{t+h})
-\Phi(\rho^\kappa_t)\|_{L^2(0,1)}^2\,dt
\leq C_{\mathrm{time}}\sqrt h
\end{equation*}
for every $h\in(0,T)$.
The constant $C_{\mathrm{time}}$ is
independent of $\kappa$ and $h$.

Now, we establish the relative compactness
of $\{\Phi(\rho^\kappa)\}_{\kappa\in\mathbb N}$
in $L^2((0,T)\times(0,1))$.
By \eqref{eq:uniform_constants}, the family
is uniformly bounded in
$L^2(0,T;\mathcal H^1(0,1))$.
Moreover, by
\eqref{eq:time_translations_bound},
\[
\begin{aligned}
&\sup_\kappa
\|\Phi(\rho^\kappa(\cdot+h))
-\Phi(\rho^\kappa)\|_
{L^2(0,T-h;L^2(0,1))}^2
\\
&\quad\leq C_{\mathrm{time}}\sqrt h
\longrightarrow 0
\qquad\text{as }h\to0.
\end{aligned}
\]
Since the embedding
\[
\mathcal H^1(0,1)
\hookrightarrow\hookrightarrow L^2(0,1)
\]
is compact, the compactness criterion of
Simon (see \cite[Theorem 3]{Simon1987}) implies that
$\{\Phi(\rho^\kappa)\}_{\kappa\in\mathbb N}$
is relatively compact in
$L^2(0,T;L^2(0,1))$,
which is precisely
$L^2((0,T)\times(0,1))$.
\end{proof}

Finally, the four lemmas that follow are adaptations of \cite[Propositions 3.1, 3.2, 3.3, and 3.5]{dgn21}, respectively, and are needed to show that limit points $\rho^*$ of $\{\rho^\kappa\}_{\kappa>0}$ belong to the correct spaces and satisfy the appropriate integral formulation. We follow the proofs therein but with some modifications. {{
Since the interacting particle system needed to obtain the equation in Definition \ref{def:weak_robin} has  unbounded occupation variables, we need to invoke the comparison principle (see Proposition \ref{prop:comparison}) to show that $\{\rho^\kappa\}_\kappa$ is bounded in $L^\infty([0,T]\times[0,1])$.
}}

\begin{lemma}\label{lemma:Phi*_L2_H1}
If $\Phi^*$ is a strong limit point of the sequence $\{\Phi(\rho^\kappa)\}_{\kappa>0}$ in $L^2([0, T]\times[0, 1])$, then $\Phi^*\in {{L^2(0,T; \mathcal{H}^1)}}$. 
\end{lemma}
\begin{proof} Let $\{\Phi(\rho^{\kappa_j})\}_{\in\mathbb{N}}$ be a subsequence of $\{\Phi(\rho^\kappa)\}_{\kappa>0}$ which converges to $\Phi^*$ strongly in $L^2([0, T]\times[0, 1])$ as $j\to\infty$.  Let $H\in C_c^{0, \infty}([0, T]\times(0, 1))$: by Theorems \ref{thm:energy_estimate} and \ref{thm:energy_properties}, we have that
\begin{equation*}
    \llangle \Phi(\rho^{\kappa_j}), \partial_u H\rrangle = - \llangle \partial_u\Phi(\rho^{\kappa_j}), H\rrangle
\end{equation*}
for all $j\in\mathbb{N}$, and hence, by the Cauchy-Schwarz inequality,
\begin{equation}\label{eq:limJ_partialPhi}
    \mylim_{j\to\infty}  \llangle \partial_u\Phi(\rho^{\kappa_j}), H\rrangle
    = \mylim_{j\to\infty} - \llangle \Phi(\rho^{\kappa_j}), \partial_u H\rrangle = -\llangle {{ \Phi^*}}, \partial_u H\rrangle.
\end{equation}
On the other hand, by Lemma \ref{lemma:Phi_L2_H1}, the sequence $\{\partial_u\Phi(\rho^{\kappa_j})\}_{j\in\mathbb{N}}$ is bounded in $L^2([0, T]\times[0, 1])$, and thus, there exist\footnote{This is a consequence of the Banach-Alaoglu Theorem and the fact that the space $L^2([0, T]\times[0, 1])$ is reflexive.} a subsequence  $\{\partial_u\Phi(\rho^{\kappa_{j_l}})\}_{l\in\mathbb{N}}$ and a function $\Psi\in L^2([0, T]\times[0, 1])$ such that $\{\partial_u\Phi(\rho^{\kappa_{j_l}})\}_{l\in\mathbb{N}}$ converges \textit{weakly} to $\Psi$ in $L^2([0, T]\times[0, 1])$ as $l\to\infty$. In particular, calling {{$\Psi=:\partial_u \Phi^*$}}, we have that
\begin{equation}\label{eq:limJ_partialPhi_2}
    \mylim_{l\to\infty} \llangle \partial_u \Phi(\rho^{\kappa_{j_l}}),  H\rrangle  = \llangle {{\partial_u \Phi^*}}, H\rrangle.
\end{equation}
Thus, by \eqref{eq:limJ_partialPhi}, we get
\begin{equation*}
    \llangle {{ \Phi^*}}, \partial_u H\rrangle=-\llangle {{\partial_u \Phi^*}}, H\rrangle.
\end{equation*}
But now, a function $\zeta$ is in ${{L^2(0,T; \mathcal{H}^1)}}$ if and only if there exists $\partial_u\zeta \in L^2([0, T]\times[0, 1])$ such that $\llangle \zeta, \partial_u H\rrangle=-\llangle \partial_u\zeta, H\rrangle$ for all $H\in C_c^{0, \infty}([0, T]\times(0, 1))$. This completes the proof.
\end{proof}

\begin{lemma}\label{lemma:Phi*_weak} Assume that the sequence $\{\rho^{\kappa_j}\}_{j\in\mathbb{N}}$ converges to $\rho^*$ strongly in $L^2([0, T]\times[0, 1])$ as $j\to\infty$. Then,
\begin{equation*}
    \mylim_{j\to\infty}\int_0^t \left\langle \partial_u \Phi(\rho_s^{\kappa_j}), \partial_u H_s\right\rangle \de s = \int_0^t \left\langle \partial_u \Phi(\rho_s^*), \partial_u H_s\right\rangle \de s
\end{equation*}
for all $t\in [0, T]$ and every $H\in C^{1, 2}([0, T]\times[0, 1])$.
\end{lemma}
\begin{proof}
We start by noting that it suffices to prove that
\begin{equation*}
    \mylim_{j\to\infty}\int_0^t \left\langle \partial_u \Phi(\rho_s^{\kappa_j}), G_s\right\rangle \de s = \int_0^t \left\langle \partial_u \Phi(\rho_s^*), G_s\right\rangle \de s
\end{equation*}
for all $t\in [0, T]$ and every $G\in C^{1, 1}([0, T]\times[0, 1])$. {{By the hypothesis of this lemma and the fact that $\Phi$ is Lipschitz, $\Phi(\rho^*)$ is a strong limit point of $\{\Phi(\rho^\kappa)\}_{\kappa>0}$ in $L^2([0,T]\times[0,1])$. Therefore, we can apply the result above.
}}  By \eqref{eq:limJ_partialPhi_2}, we can extract a subsequence $\{\partial_u\Phi(\rho^{\kappa_{j_l}})\}_{l\in\mathbb{N}}$ such that
\begin{equation}\label{eq:limJ_G}
    \mylim_{l\to\infty}\int_0^T \left\langle \partial_u \Phi(\rho_s^{\kappa_{j_l}}), \bar G_s\right\rangle \de s = \int_0^T \left\langle \partial_u \Phi(\rho_s^*), \bar G_s\right\rangle \de s
\end{equation}
for every $\bar G\in C_c^{0, \infty}([0, T]\times(0, 1))$. 

Let now $\varphi_\delta:[0, T]\to[0, 1]$ be a continuous map such that $\varphi_\delta(s)=1$ for $s\in[0, t]$ and $\varphi_\delta(s)=0$ for $s\in [t+\delta, T]$. If $G\in C_c^{1, \infty}([0, T]\times(0, 1))$, then $\bar G_s(u):=\varphi_\delta(s)G_s(u)$ is in $C_c^{0, \infty}([0, T]\times(0, 1))$. By \eqref{eq:limJ_G}, we see that
\begin{align*}
    &\mylim_{j\to\infty}\int_0^t \left\langle \partial_u \Phi(\rho_s^{\kappa_j}), G_s\right\rangle \de s  +\mylim_{j\to\infty}\int_t^{t+\delta} \left\langle \partial_u \Phi(\rho_s^{\kappa_j}), G_s\right\rangle \varphi_\delta(s) \de s 
    \\&=\mylim_{j\to\infty}\int_0^T \left\langle \partial_u \Phi(\rho_s^{\kappa_j}), G_s\right\rangle \varphi_\delta(s) \de s 
    \\&=\int_0^T \left\langle \partial_u \Phi(\rho_s^*), G_s\right\rangle \varphi_\delta(s) \de s 
    \\&=\int_0^t \left\langle \partial_u \Phi(\rho_s^*), G_s\right\rangle \de s  +\int_t^{t+\delta} \left\langle \partial_u \Phi(\rho_s^*), G_s\right\rangle \varphi_\delta(s) \de s
\end{align*}
for any $G\in C_c^{1, \infty}([0, T]\times(0, 1))$ and $\delta>0$. By the Cauchy-Schwarz inequality and Lemma \ref{lemma:Phi_L2_H1}, the integrals on $[t, t+\delta]$ above vanish as $\delta\to0$, and hence
\begin{equation*}
    \mylim_{j\to\infty}\int_0^t \left\langle \partial_u \Phi(\rho_s^{\kappa_j}), G_s\right\rangle \de s = \int_0^t \left\langle \partial_u \Phi(\rho_s^*), G_s\right\rangle \de s 
\end{equation*}
for any $G\in C_c^{1, \infty}([0, T]\times(0, 1))$.
By the density of $C_c^{1,\infty}([0,T]\times(0,1))$
in $L^2([0,T]\times(0,1))$, the uniform $L^2$ bound on
$\partial_u\Phi(\rho^{\kappa_j})$ provided by
Lemma~\ref{lemma:Phi_L2_H1}, and the Cauchy--Schwarz
inequality, the convergence above extends to every
$G\in L^2([0,T]\times(0,1))$. In particular, it holds
for every $G\in C^{1,1}([0,T]\times[0,1])$.
For a similar density argument, see
\cite[proof of Proposition~5.5]{FGN14}.
This completes the proof.
\end{proof}

\begin{lemma}\label{lemma:rho*_boundary} Assume that the sequence $\{\rho^{\kappa_j}\}_{j\in\mathbb{N}}$ converges to $\rho^*$ strongly in $L^2([0, T]\times[0, 1])$ as $j\to\infty$, and that $\kappa_j\to\infty$ as $j\to\infty$. Then, $\Phi(\rho^*_t(0))=\frac{\alpha}{\lambda}$ and $\Phi(\rho^*_t(1))=\frac{\beta}{\delta}$ for almost every $t\in(0, T]$.
\end{lemma}
\begin{proof} We will only show that $\Phi(\rho^*_t(0))=\frac{\alpha}{\lambda}$, as the proof of the second equality is completely analogous. By the inequality $(x+y+z)^2\le 3(x^2+y^2+z^2)$, we see that
\begin{align}
    \int_0^T\big\{\Phi(\rho^{\kappa_j}_t(0))-&\Phi(\rho^*_t(0))\big\}^2 \de t\le 3 \int_0^T \left\{\Phi(\rho^{\kappa_j}_t(0))-\frac{1}{\eps}\int_0^\eps \Phi(\rho^{\kappa_j}_t(u))\de u\right\}^2 \de t\label{eq:3_a}
    \\&\phantom{\le} +3 \int_0^T \left\{\frac{1}{\eps}\int_0^\eps \Phi(\rho^{\kappa_j}_t(u))\de u - \frac{1}{\eps}\int_0^\eps \Phi(\rho^*_t(u))\de u\right\}^2 \de t\label{eq:3_b}
    \\&\phantom{\le} +3 \int_0^T \left\{\Phi(\rho^*_t(0)) - \frac{1}{\eps}\int_0^\eps \Phi(\rho^*_t(u))\de u \right\}^2 \de t\label{eq:3_c}
\end{align}
for any $j\in\mathbb{N}$ and any $\eps>0$. 

Now, by \eqref{eq:partialPhi_bdd}, we have that
\begin{equation}\label{eq:partialPhi_M0}
    \llangle \partial_u \Phi(\rho^{\kappa_j}), \partial_u \Phi(\rho^{\kappa_j})\rrangle \le C
\end{equation}
for $C<\infty$ independent of $j$. But then, since
\begin{align*}
    \left|\Phi(\rho^{\kappa_j}_t(0))-\frac{1}{\eps}\int_0^\eps \Phi(\rho^{\kappa_j}_t(u))\de u\right| & = \left|\frac{1}{\eps}\int_0^\eps\int_0^u \partial_v \Phi(\rho^{\kappa_j}_t(v))\de v\de u\right|
    \\& = \left|\frac{1}{\eps}\int_0^\eps \partial_u \Phi(\rho^{\kappa_j}_t(u))(\eps-u)\de u\right|
    \\& \le \sqrt{\frac{\eps}{3}}\llangle \partial_u \Phi(\rho^{\kappa_j}), \partial_u \Phi(\rho^{\kappa_j})\rrangle^{\frac{1}{2}},
\end{align*}
we obtain
\begin{equation*}
    \eqref{eq:3_a}\le \eps \llangle \partial_u \Phi(\rho^{\kappa_j}), \partial_u \Phi(\rho^{\kappa_j})\rrangle \le \eps C.
\end{equation*}
Defining $\partial_u\Phi(\rho^*)$ as above equation \eqref{eq:limJ_partialPhi_2}, taking the limit as $j\to\infty$ in \eqref{eq:partialPhi_M0} and applying the exact same argument as above, we also get
\begin{equation*}
    \eqref{eq:3_c}\le \eps \llangle \partial_u \Phi(\rho^*), \partial_u \Phi(\rho^*)\rrangle \le\eps C.
\end{equation*}
Finally, recalling \eqref{eq:Phi_g*} and by the Cauchy-Schwarz inequality, we have that
\begin{equation*}
    \eqref{eq:3_b}\le \frac{3 (g^*)^2}{\eps} \llangle \rho^{\kappa_j}-\rho^*, \rho^{\kappa_j}-\rho^*\rrangle.
\end{equation*}

Combining everything together, we get
\begin{equation*}
    \int_0^T\big\{\Phi(\rho^{\kappa_j}_t(0))-\Phi(\rho^*_t(0))\big\}^2\le 8c\, \Phi(c_0) M_0\,\eps+\frac{3 (g^*)^2}{\eps} \llangle \rho^{\kappa_j}-\rho^*, \rho^{\kappa_j}-\rho^*\rrangle.
\end{equation*}
Sending $j\to\infty$ and then $\eps\to0$, we have thus proved that
\begin{equation}\label{eq:limJ_Phi0}
    \mylim_{j\to\infty}\int_0^T\big\{\Phi(\rho^{\kappa_j}_t(0))-\Phi(\rho^*_t(0))\big\}^2\de t=0.
\end{equation}
Now, using \eqref{eq:boundary_conditions}, recalling the definition of $\llangle \cdot, \cdot \rrangle_\kappa$ in Section \ref{subsec:main_results}, and by the same argument that led to \eqref{eq:partialPhi_bdd}, we see that
\begin{align*}
    \int_0^T\left\{\Phi(\rho^{\kappa_j}_t(0))-\frac{\alpha}{\lambda}\right\}^2\de t &=  \int_0^T\left\{\frac{\partial_u \Phi(\rho_t^{\kappa_j}(0))} {\kappa_j\lambda}\right\}^2 \de t 
    \\&\le \frac{1}{\kappa_j\lambda} \llangle \partial_u \Phi(\rho^{\kappa_j}), \partial_u \Phi(\rho^{\kappa_j})\rrangle_{\kappa_j}
    \\&\le \frac{C}{\kappa_j \lambda},
\end{align*}
where $C<\infty$ is independent of $j$. Since $\kappa_j\to\infty$ as $j\to\infty$, this yields
\begin{equation*}
    \mylim_{j\to\infty}\int_0^T\left\{\Phi(\rho^{\kappa_j}_t(0))-\frac{\alpha}{\lambda}\right\}^2\de t=0.
\end{equation*}
Combining this with \eqref{eq:limJ_Phi0} and using the inequality $(x+y)^2\le 2(x^2+y^2)$, we have thus proved that
\begin{equation*}
    \int_0^T\left\{\Phi(\rho^*_t(0))-\frac{\alpha}{\lambda}\right\}^2\de t=0,
\end{equation*}
which implies that $\Phi(\rho^*_t(0))=\frac{\alpha}{\lambda}$ for almost every $t\in(0, T]$, as required.
\end{proof}

\begin{lemma}\label{lemma:rho*_weak} Assume that the sequence $\{\rho^{\kappa_j}\}_{j\in\mathbb{N}}$ converges to $\rho^*$ strongly in $L^2([0, T]\times[0, 1])$ as $j\to\infty$. Then, there exists a subsequence $\{\rho^{\kappa_{j_l}}\}_{l\in\mathbb{N}}$ such that
\begin{equation*}
    \mylim_{l\to\infty}\left\langle \rho_t^{\kappa_{j_l}}, H_t\right\rangle = \left\langle \rho_t^*, H_t\right\rangle
\end{equation*}
for all $t\in[0, T]$ and every $H\in C^{0, 2}([0, T]\times[0, 1])$.
\end{lemma}

\begin{proof} The proof closely follows that of \cite[Proposition 3.5]{dgn21}, so we refer to the argument therein. We highlight that this argument \textit{does} require the boundedness of $\{\rho^\kappa\}_\kappa$ in $L^\infty([0, T]\times[0, 1])$: indeed, calling $\mathcal{F}_\kappa(t, H_t):=\langle \rho_t^\kappa, H_t\rangle$ for a fixed $H\in C^{0, 2}([0, T]\times[0,1])$, the boundedness of $\{\rho^\kappa\}_\kappa$ is $L^p([0, T]\times[0, 1])$ is not enough to guarantee a uniform bound for $|\mathcal{F}_\kappa(t, H_t)|$, which is necessary to apply the Arzelà-Ascoli theorem.\end{proof}

\appendix\section{Weak Solutions of the Nonlinear PDE}\label{sec:appendix}
In this section, we prove two fundamental properties of our nonlinear PDE. We start by proving uniqueness of weak solutions with all three type of boundary conditions.
\begin{proposition}[Uniqueness of Weak Solutions]\label{prop:uniqueness_zre} There exists a unique weak solution to \eqref{eq:zre_robin}, \eqref{eq:zre_neumann} and \eqref{eq:zre_dirichlet} in the sense of Definitions \ref{def:weak_robin}, \ref{def:weak_neumann} and \ref{def:weak_dirichlet}, respectively.
\end{proposition}

\begin{proof} The proof of uniqueness of weak solutions to \eqref{eq:zre_robin} and \eqref{eq:zre_neumann} follows from \cite[Lemma 2.10]{agnr25}, so we only need to consider the Dirichlet case. We follow \cite[Section 7.1]{bdgn20}. To that end, assume that $\rho_t^1$ and $\rho_t^2$ are both weak solutions to \eqref{eq:zre_dirichlet} (with the same initial condition $\gamma$), and let
\begin{equation*}
    \omega_t(u) := \rho^1_t(u) - \rho^2_t(u).
\end{equation*}
Subtracting the integral formulation satisfied by 
$\rho^1$ and $\rho^2$ and performing an integration by parts, we obtain that, for all 
$G\in C_0^{1,2}([0,T]\times[0,1])$,
\begin{equation}\label{eq:weak_difference}
    \langle w_T, G_T\rangle +\int_0^T \langle \partial_u \Phi(\rho_s^1)-\partial_u\Phi(\rho_s^2), \partial_u G_s\rangle \de s - \int_0^T \langle w_s, \partial_s G_s\rangle \de s = 0.
\end{equation}
Note that, by a standard smooth approximation procedure, \eqref{eq:weak_difference} is well defined even for functions $G:[0, T]\times[0, 1]\to\mathbb{R}$ such that $G\in {{L^2(0,T; \mathcal{H}^1)}}$ and $\partial_t G\in L^2([0, T]\times[0, 1])$. 
Define now the map $\theta:[0, T]\times[0, 1]\to\mathbb{R}$ via
\begin{equation}
    \theta_t(u):= 
    \int_t^T\big\{\Phi(\rho^1_s(u))-\Phi(\rho^2_s(u))\big\}\de s.
\end{equation}
Since both solutions satisfy $\Phi(\rho^1_t(0))=\Phi(\rho^2_t(0))=\frac{\alpha}{\lambda}$ and $
\Phi(\rho^1_t(1))=\Phi(\rho^2_t(1))=\frac{\beta}{\delta}$ 
it follows that $\theta_t(0)=\theta_t(1)=0$ for almost every $t\in [0, T]$. Moreover, since $\Phi(\rho)\in {{L^2(0,T; \mathcal{H}^1)}}$, we have that
\begin{align*}
    &\partial_t \theta_t(u) = 
    \Phi(\rho^2_s(u))-\Phi(\rho^1_s(u)) \in L^2([0, T]\times[0, 1]),
    \\&\partial_u\theta_t(u) = \int_t^T \big\{\partial_u \Phi(\rho_s^1(u))-\partial_u\Phi(\rho_s^2(u))\big\}\de s \in L^2([0, T]\times[0, 1]).
\end{align*}
Hence, since $\theta_T(u)=0$ for all $u\in[0, 1]$, using $\theta$ as a test function in \eqref{eq:weak_difference} yields
\begin{equation*}\label{eq:weak_diff_theta}
    \int_0^T \langle \partial_u \Phi(\rho_s^1)-\partial_u\Phi(\rho_s^2), \partial_u \theta_s\rangle \de s - \int_0^T \langle w_s, \partial_s \theta_s\rangle \de s = 0,
\end{equation*}
which we can rewrite as
\begin{equation*}
    \begin{split}
    \int_0^1\int_0^T&\Bigg\{\omega_t(u) \big(\Phi(\rho^1_s(u))-\Phi(\rho^2_s(u))\big)
    \\&+\big(\partial_u \Phi(\rho_s^1)-\partial_u\Phi(\rho_s^2)\big)\int_t^T \big\{\partial_u \Phi(\rho_s^1(u))-\partial_u\Phi(\rho_s^2(u))\big\}\de s\Bigg\}\de t\de u = 0,
    \end{split}
\end{equation*}
that is,
\begin{equation*}
    \int_0^T\langle \omega_t, \Phi(\rho^1_t)-\Phi(\rho^2_t)\rangle\de t + \frac{1}{2}\int_0^1\left\{\int_0^T \big(\partial_u \Phi(\rho_s^1)-\partial_u\Phi(\rho_s^2)\big)\right\}^2\de u = 0. 
\end{equation*}
Since $\Phi$ is strictly increasing, this implies that $\omega_t(u)=0$ almost everywhere on $[0, T]\times[0, 1]$.
\end{proof}

The next proposition states that any weak solution to \eqref{eq:zre_robin} is such that, for any fixed test function $H$ in an appropriate space, the map $[0, T]\ni t\mapsto \langle \rho_t^\kappa, H_t\rangle$ is $\frac{1}{2}$-Hölder continuous in time, which allows us to prove Lemma \ref{lemma:rho*_weak}. The statement is reminiscent of \cite[Proposition 3.4]{dgn21}; however, our proof only requires the densities $\{\rho^\kappa\}_{\kappa>0}$ to be bounded in $L^2([0, T]\times[0, 1])$ and \textit{not} in $L^\infty([0, T]\times[0, 1])$, so we will need some additional arguments to prove the result.
\begin{proposition}\label{prop:time_reg} 
For every
\begin{enumerate}[i)]
    \item $H\in C^{1, 2}([0, T]\times [0, 1])$ if $0<\kappa\le1$, and
    \item $H\in C_0^{1, 2}([0, T]\times [0, 1])$ if $\kappa>1$,
\end{enumerate}
there exists a constant $C$ independent of $\kappa$ such that 
\begin{equation*}
    \left|\langle \rho_t^\kappa, H_t\rangle - \langle \rho_s^\kappa, H_s\rangle\right| \le C|t-s|^\frac{1}{2}
\end{equation*}
for all $s, t\in [0, T]$.
\end{proposition}
\begin{proof} i) Assume $0<\kappa\le1$ and let $H\in C^{1, 2}([0, T]\times [0, 1])$. By performing an integration by parts, we can rewrite \eqref{eq:integral_zre_rob} as
\begin{equation} \label{eq:integral_ibp}
    \begin{split}
    \langle\rho^\kappa_t, H_t\rangle &- \langle\gamma, H_0\rangle - \int_0^t\langle \rho^\kappa_s, \partial_s H_s\rangle \de s +\int_0^t \langle\partial_u \Phi(\rho^\kappa_s), \partial_u H_s\rangle\de s
    \\&-\kappa\int_0^t \big\{(\beta-\delta\Phi(\rho^\kappa_s(1)))H_s(1) - (\lambda \Phi(\rho^\kappa_s(0))-\alpha)H_s(0) \big\}\de s = 0.
    \end{split}
\end{equation}
Subtracting \eqref{eq:integral_ibp} from the same equation evaluated at a different time $s\in[0, T]$, we immediately obtain
\begin{equation}\label{eq:time_diff_bound}
    \begin{split}
       |\langle\rho^\kappa_t, H_t\rangle &-\langle\rho^\kappa_s, H_s\rangle|\le \left|\int_s^t\langle \rho^\kappa_r, \partial_r H_r\rangle \de r\right|+\left|\int_s^t \langle\partial_u \Phi(\rho^\kappa_r), \partial_u H_r\rangle\de r\right|
       \\&+\kappa \left| \int_s^t \big\{(\beta-\delta\Phi(\rho^\kappa_r(1)))H_r(1) - (\lambda \Phi(\rho^\kappa_r(0))-\alpha)H_r(0) \big\}\de r\right|.
    \end{split}
\end{equation}
Now, by the Cauchy-Schwarz inequality,
\begin{equation*}
    \left|\int_s^t\langle \rho^\kappa_r, \partial_r H_r\rangle \de r\right|\le \|\rho^\kappa\|_{L^2([0, T]\times[0, 1])} \|\partial_t H\|_\infty \sqrt{|t-s|}.
\end{equation*}
and similarly,
\begin{equation*}
    \left|\int_s^t \langle\partial_u \Phi(\rho^\kappa_r), \partial_u H_r\rangle\de r\right|\le \|\partial_u\Phi(\rho^\kappa)\|_{L^2([0, T]\times[0, 1])} \|\partial_u H\|_\infty \sqrt{|t-s|}.
\end{equation*}
To treat the last line of \eqref{eq:time_diff_bound}, we apply a similar argument to the one right below equation \eqref{eq:integral_ibp_1}: by the standard inequality $\|f\|_{L^\infty([0, 1])}\le \|f\|_{L^2([0, 1])}+\|\partial_u f\|_{L^2([0, 1])}$ applied to $f=\Phi(\rho^\kappa_r)$, since $0<\kappa\le 1$ and by the Cauchy-Schwarz inequality, we see that the last term on the right-hand side of \eqref{eq:time_diff_bound} is bounded from above by
\begin{align*}
    \kappa &\left| \int_s^t \big\{(\beta-\delta\Phi(\rho^\kappa_r(1)))H_r(1) - (\lambda \Phi(\rho^\kappa_r(0))-\alpha)H_r(0) \big\}\de r\right|
    \\&\le (\alpha+\beta) \|H\|_\infty |t-s|
    \\&\phantom{\le}+(\lambda+\delta)\|H\|_\infty \left|\int_s^t \left\{\|\Phi(\rho^\kappa_r)\|_{L^2([0, 1])}+\|\partial_u\Phi(\rho^\kappa_r)\|_{L^2([0, 1])}\right\}\de r\right|
    \\&\le (\alpha+\beta) \|H\|_\infty \sqrt{2T}\sqrt{|t-s|}+(\lambda+\delta)\|H\|_\infty \sqrt{|t-s|} \|\Phi(\rho^\kappa)\|_{{{L^2(0,T; \mathcal{H}^1)}}}.
\end{align*}
But now, by Lemmas \ref{lemma:rho_Lp} and \ref{lemma:Phi_L2_H1}, the coefficients 
\begin{equation*}
    \|\rho^\kappa\|_{L^2([0, T]\times[0, 1])}, \quad \|\partial_u\Phi(\rho^\kappa)\|_{L^2([0, T]\times[0, 1])} \quad \text{and} \quad \|\Phi(\rho^\kappa)\|_{{{L^2(0,T; \mathcal{H}^1)}}}
\end{equation*} can all be bounded from above by a constant independent of $\kappa$, which yields the required bound.

\vspace{1em}
ii) Assume now $\kappa>1$ and let $H\in C_0^{1, 2}([0, T]\times [0, 1])$. Since $H_s(0)=H_s(1)=0$ for all $s$, the last line of \eqref{eq:integral_ibp} vanishes. But then, we can run the same argument as above to estimate the first line of \eqref{eq:time_diff_bound}. This completes the proof.
\end{proof}

Finally, the proposition that follows guarantees that, if the initial condition $\gamma$ of \eqref{eq:zre_robin} is bounded, then the family $\{\rho_\kappa\}_\kappa$ is bounded in $L^\infty([0, T]\times[0, 1])$.
\begin{proposition}[Comparison Principle]\label{prop:comparison} Assume that $\gamma(u)\le c_\gamma<\infty$ for any $u\in [0, 1]$, and let $\rho^\kappa_t(u)$ denote the unique weak solution to \eqref{eq:zre_robin}. Define
\begin{equation*}
    c_0:=\mymax \left\{c_\gamma, \Phi^{-1}\left(\frac{\alpha}{\lambda}\right), \Phi^{-1}\left(\frac{\beta}{\delta}\right)\right\}.
\end{equation*}
Then, for each $\kappa>0$, $\rho^\kappa\le c_0$ almost everywhere on $[0, T]\times[0, 1]$. In particular, the family $\{\rho^\kappa\}_\kappa$ is bounded in $L^\infty([0, T]\times[0, 1])$.
\end{proposition}

Before proving Proposition \ref{prop:comparison}, we recall some results from \cite{agnr25}, namely Lemmas A.2, A.3 and A.4.

\begin{lemma}\label{lemma:parabolic_uniqueness}
Suppose that: \begin{enumerate}[i)] 
\item $a=a(t, u)$ is a positive $C^{2,2}([0,T]\times [0,1])$ function, 
\item $b=b(t, u)$ is a function defined on $[0, T]\times[0, 1]$ such that $b(\cdot, 0)$ and $b(\cdot, 1)$ are in $C^2([0,T])$,
\item $h=h(u)$ is a  $C^2_0([0,1])$ function.
\end{enumerate}
Then, for each $t \in (0,T]$, the problem with Robin boundary conditions 
\begin{equation}\label{eq:parabolic_pde}
    \left\{
    \begin{array}{ll} 
    \partial_s\phi + a\Delta \phi  =0 &  \mbox{ for } (s,u) \in [0,t)\times(0,1),
    \\\partial_u\phi(s, 0)=b(s, 0)\,\phi(s, 0) & \mbox{ for } s \in [0,t),
    \\ \partial_u\phi(s, 1) = -b(s, 1)\,\phi(s, 1) & \mbox{ for } s \in [0,t),
    \\ \phi(t, u)  = h(u) &\mbox{ for } u \in (0,1)
    \end{array}\right.
\end{equation}
has a unique solution $\phi_0$ in $C^{1,2}([0,t]\times[0,1])$. Moreover, if $0 \le h\le1$, then 
\begin{equation*}
    0\le \phi_0(s, u) \le 1  \ \mbox{ for } (s,u) \in [0,t] \times [0,1]. 
\end{equation*}
\end{lemma}

\begin{lemma}\label{lemma:parabolic_bound}
Let $\phi_0$ be the unique solution of the parabolic equation \eqref{eq:parabolic_pde}. There exists a positive constant $C=C(b,h)$ such that 
\begin{equation*}
    \int_0^t\int_0^1 a(s, u)(\Delta \phi_0(s, u))^2\de u\de s \le C(b,h).
\end{equation*}
\end{lemma}

\begin{lemma}\label{lemma:convergent_sequence}
Let $a$ be a non-negative, bounded, measurable function on $[0,T] \times [0,1]$. There exists a sequence $\{a_k\}_{k\ge0}$ of positive functions in $C^{\infty}([0,T]\times [0,1])$ such that, for any $\omega\in L^2([0,T]\times[0,1])$,
\begin{equation*}
    \frac{1}{k}\le a_k \le \|a\|_{L^{\infty}}+\frac{1}{k} \ \ \text{and } \ \ \left\|\omega\;\frac{a-a_k}{\sqrt{a_k}}\right\|_{L^2([0,T]\times[0,1])} \to 0, 
\end{equation*}
as $k\to\infty$.
\end{lemma}

\begin{proof}[Proof of Proposition \ref{prop:comparison}] Let 
\begin{equation*}
    \omega_t^\kappa(u) := \rho_t^\kappa(u)-c_0.
\end{equation*}
From \eqref{eq:integral_zre_rob}, we see that 
\begin{equation}\label{eq:omega_formulation} \begin{split}
    \langle \omega_t^\kappa , G_t\rangle &-\int_0^t \big\{\langle \omega_s^\kappa, \partial_s G_s \rangle + \big\langle \Phi(\rho_s^\kappa)-\Phi(c_0), \Delta G_s\big\rangle\big\}\de s
    \\& +\int_0^t (\Phi(\rho_s^\kappa(1))-\Phi(c_0))\,\big\{\partial_uG_s(1) +\kappa \delta G_s(1)\big\}\de s
    \\&-\int_0^t (\Phi(\rho_s^\kappa(0))-\Phi(c_0))\,\big\{ \partial_uG_s(0)- \kappa \lambda G_s(0)\big\} \de s
    \\&= \langle \gamma-c_0, G_0\rangle +\kappa\int_0^t \big\{(\beta-\delta\Phi(c_0)) G_s(1)-(\lambda\Phi(c_0)-\alpha) G_s(0)\big\}\de s.
    \end{split}
\end{equation}
By \eqref{eq:Phi_g*}, we have that $|\Phi(\rho_s^\kappa(u))-\Phi(c_0)|\le g^*|\omega_s^\kappa(u)|$ for all $u\in[0,1]$ and $s\in[0, T]$. Then, we can define the bounded function 
\begin{equation*}
    v_s^\kappa(u):=\frac{\Phi(\rho_s^\kappa(u))-\Phi(c_0)}{\rho_s^\kappa(u)-c_0}\ \ \text{ for }u\in[0,1]\text{ and }s\in[0, T],
\end{equation*}
so that
\begin{equation*}
    \Phi(\rho_s^\kappa(u))-\Phi(c_0) = v_s^\kappa(u)\,\omega_s^\kappa(u) \ \ \text{ for }u\in[0,1]\text{ and }s\in[0, T].
\end{equation*}
Using the function $v_t^\kappa$, we can rewrite  \eqref{eq:omega_formulation} as
\begin{equation}\label{eq:omega_formulation_2} 
\begin{split}
    \langle \omega_t^\kappa , G_t\rangle &-\int_0^t \big\{\langle \omega_s^\kappa, \partial_s G_s \rangle + \big\langle \Phi(\rho_s^\kappa)-\Phi(c_0), \Delta G_s\big\rangle\big\}\de s
    \\& +\int_0^t v_s^\kappa(1)\omega_s^\kappa(1)\,\big\{\partial_uG_s(1) +\kappa \delta G_s(1)\big\}\de s
    \\&-\int_0^t v_s^\kappa(0)\omega_s^\kappa(0)\,\big\{ \partial_uG_s(0)- \kappa \lambda G_s(0)\big\} \de s
    \\&= \langle \gamma-c_0, G_0\rangle +\kappa\int_0^t \big\{(\beta-\delta\Phi(c_0)) G_s(1)-(\lambda\Phi(c_0)-\alpha) G_s(0)\big\}\de s.
    \end{split}
\end{equation}
To estimate the integrals in \eqref{eq:omega_formulation_2}, we employ a suitable test function, namely the solution to the parabolic equation \eqref{eq:parabolic_pde}. Since the function $\Phi$ is strictly increasing, it follows that $v^\kappa$ is non negative, and by \eqref{eq:Phi_g*} we also have that $v^\kappa\le g^*$. However, the function $v^\kappa$ does not possess the required regularity: to address this, we apply Lemma \ref{lemma:convergent_sequence} with $a=v^\kappa$, which guarantees the existence of a sequence of functions $\{a_n^\kappa\}_{n\ge 0}$ which are $C^{\infty}$ in both space and time such that
\begin{equation*}
    \frac{1}{n} \le a_n^\kappa \le g^*+\frac{1}{n} \ \ \ \text{and} \ \ \ \left\|\omega^\kappa\, \frac{a_n^\kappa-v^\kappa}{\sqrt{a_n^\kappa}}\right\|_{L^2([0,T] \times [0,1])} \to 0
\end{equation*}
as $n\to\infty$. Note that, as $\{\rho^\kappa\}_\kappa$ is bounded in $L^2([0,T] \times [0,1])$, $\{\omega^\kappa\}_\kappa$ is also bounded in $ L^2([0,T] \times [0,1])$. For fixed $h \in C^2_0([0,1])$ such that $0\le h\le 1$, consider the parabolic problem \eqref{eq:parabolic_pde} with $a$ replaced by $a_n^\kappa$ and with $b(s,0)=\kappa \lambda$, $b(s,1)=\kappa\delta$. Then, from Lemma \ref{lemma:parabolic_uniqueness}, there exists a unique solution $\phi_n^\kappa$ to this problem. Let us now choose the test function in \eqref{eq:omega_formulation_2} to be $G=\phi_n^\kappa$. Due to the boundary conditions satisfied by $\phi_n^\kappa$, the second and third line of \eqref{eq:omega_formulation_2} vanish, and thus we obtain 
\begin{equation}\label{eq:omega_formulation_3}
    \begin{split}
    &\langle \omega_t^\kappa , (\phi_n^\kappa)_t\rangle - \int_0^t \langle \omega_s, \partial_s \phi^n_s+  v_s \Delta \phi^n_s\rangle\de s
    \\& = \langle \gamma-c_0, (\phi_n^\kappa)_0\rangle +\kappa\int_0^t \big\{(\beta-\delta\Phi(c_0)) (\phi_n^\kappa)_s(1)-(\lambda\Phi(c_0)-\alpha) (\phi_n^\kappa)_s(0)\big\}\de s.
    \end{split}
\end{equation}
Since $(\phi_n^\kappa)_t=h$, we have that $\langle \omega_t^\kappa , (\phi_n^\kappa)_t\rangle=\langle \omega_t^\kappa, h\rangle$. We now sum and subtract the term $\langle \omega_s^\kappa, (a_n^\kappa)_s\Delta (\phi_n^\kappa)_s\rangle$ in \eqref{eq:omega_formulation_3} and use the fact that 
\begin{equation*}
    \partial_s (\phi_n^\kappa)_s + (a_n^\kappa)_s \Delta (\phi_n^\kappa)_s = 0
\end{equation*}
in the interval $(0,1)$, so that \eqref{eq:omega_formulation_3} can be rewritten as
\begin{equation}\label{eq:omega_formulation_4}
    \begin{split}
    \langle \omega_t^\kappa , &h\rangle =
    -\int_0^t \langle \omega_s^\kappa, \big((a_n^\kappa)_s-v_s^\kappa\big) \Delta (\phi_n^\kappa)_s\rangle\de s 
    \\&+\langle \gamma-c_0, (\phi_n^\kappa)_0\rangle +\kappa\int_0^t \big\{(\beta-(\delta\Phi(c_0)) (\phi_n^\kappa)_s(1)+(\alpha-\lambda\Phi(c_0)) (\phi_n^\kappa)_s(0)\big\}\de s.
   \end{split}
\end{equation}
Now, by the Cauchy-Schwarz inequality and by Lemma \ref{lemma:parabolic_bound}, and we see that
\begin{align*}
    &\left|\int_0^t \langle \omega_s^\kappa, \big((a_n^\kappa)_s-v_s^\kappa\big) \Delta (\phi_n^\kappa)_s\rangle\de s \right|
    \\&\le\left(\int_0^t \int_0^1\Big\{ \omega_s(u) \frac{v_s(u)-a_s^n(u)}{\sqrt{a^n_s(u)}}\Big\}^2\de u\de s\right)^{\frac{1}{2}} \left(\int_0^t \int_0^1 a^n_s(u)\big\{\Delta \phi^n_s(u)\big\}^2\de u\de s\right)^{\frac{1}{2}}
    \\&\le C \,\left\|\omega\,\frac{v-a^n}{\sqrt{a^n}}\right\|_{L^2([0,T]\times[0,1])},
\end{align*}
where $C$ denotes a constant depending only on $\lambda, \delta$ ad $h$. Moreover, since $0\le h\le 1$, it follows from Lemma \ref{lemma:parabolic_uniqueness} that $(\phi_n^\kappa)_s(u)\ge 0$ for all $(s, u)\in[0, t]\times[0, 1]$. Recalling the definition of $c_0$ and since $\Phi$ is nondecreasing, this implies that $\langle \gamma-c_0, (\phi_n^\kappa)_0\rangle\le 0$, $(\beta-\delta\Phi(c_0))(\phi_n^\kappa)_s(1)\le 0$ and $(\alpha-\lambda\Phi(c_0)) (\phi_n^\kappa)_s(0)\le 0$, so that the second line of \eqref{eq:omega_formulation_4} is non-positive. Combining everything together, sending $n\to\infty$ and using Lemma \ref{lemma:convergent_sequence}, it follows that 
\begin{equation}\label{eq:omega_positive}
    \langle \omega_t^\kappa, h\rangle \le 0
\end{equation}
for any $h \in C^2_0([0,1])$ such that $0\le h\le 1$. 

Consider now a sequence of functions $\{h_j\}_{j\ge1}$ in $C^2_0([0,1])$ such that 
\begin{equation*}
    \mylim_{j\to\infty}\left\|h_j - \boldsymbol{1}_{\{u \in [0,1]:\, \omega_t^\kappa(u)>0\}}\right\|_{L^2([0, 1])}=0.
\end{equation*}
Then, using \eqref{eq:omega_positive}, we obtain
\begin{equation*}
    \int_0^1 (\omega_t^\kappa)^+(u) \de u \le 0,
\end{equation*}
where $x^+ := \mymax\{x, 0\}$ denotes the positive part. Therefore, for any $t \in [0,T]$, $\rho_t^\kappa(u)\le c_0$ for almost every $u \in [0,1]$; that is, $\rho^\kappa\le c_0$ for almost every $(t,u) \in [0,T]\times[0,1]$.
\end{proof}

\begin{proof}[Proof of the Lemma \ref{lemma:inf}]
Recall that the stationary density profile is given by
\[
\bar\rho^\kappa(u)
=
\Phi^{-1}\bigl(\bar\varphi^\kappa(u)\bigr),
\]
where
\[
\bar\varphi^\kappa(u)
=
\frac{
\kappa\bigl[\alpha\delta(1-u)+\beta\lambda u\bigr]
+\alpha+\beta
}{
\kappa\lambda\delta+\lambda+\delta
}.
\]

Since $\Phi^{-1}$ is strictly increasing and continuous,
\[
\inf_{\kappa>0}\bar\rho^\kappa(u)
=
\Phi^{-1}\left(
\inf_{\kappa>0}\bar\varphi^\kappa(u)
\right).
\]

For each fixed $u\in[0,1]$, the function
$\kappa\mapsto\bar\varphi^\kappa(u)$
is monotone. Indeed,
\[
\frac{\partial}{\partial\kappa}
\bar\varphi^\kappa(u)
=
\frac{
(\alpha\delta-\beta\lambda)
\bigl[\delta-(\lambda+\delta)u\bigr]
}{
(\kappa\lambda\delta+\lambda+\delta)^2
},
\]
whose sign is independent of $\kappa$.

Consequently,
\[
\inf_{\kappa>0}\bar\varphi^\kappa(u)
=
\min\left\{
\lim_{\kappa\to0}\bar\varphi^\kappa(u),
\lim_{\kappa\to\infty}\bar\varphi^\kappa(u)
\right\}.
\]

Since
\[
\lim_{\kappa\to0}\bar\varphi^\kappa(u)
=
\frac{\alpha+\beta}{\lambda+\delta}
\]
and
\[
\lim_{\kappa\to\infty}\bar\varphi^\kappa(u)
=
\frac{\alpha}{\lambda}(1-u)
+\frac{\beta}{\delta}u,
\]
the result follows.
\end{proof}

\begin{acknowledgements} \small{O.A. acknowledges support from the National Council for Scientific and Technological Development (CNPq, Brazil) through a Junior Postdoctoral Fellowship (Grant No. 152638/2022-9), and from the Fundação para a Ciência e a Tecnologia (FCT, Portugal) under the project IT137-22-035 - UIDP/00324/2020 - Center for Mathematics of the University of Coimbra. P.G. expresses warm thanks to Fundação para a Ciência e Tecnologia FCT/Portugal for financial support through the projects UIDB/04459/2020, UIDP/04459/2020 and ERC/FCT. A.N. acknowledges support from the National Council for Scientific and Technological Development (CNPq, Brazil) through a research productivity fellowship (Grant No. 307606/2025-2). M.C.R. gratefully acknowledges support from the Dean's PhD Scholarship at Imperial College London and the kind hospitality of Instituto Superior Técnico in May 2026, when part of this work was carried out.}
\end{acknowledgements}

\bibliographystyle{Martin} 
\bibliography{references}

\end{document}